\documentclass[12pt]{amsart}
\usepackage{amssymb,amsfonts,amsmath,amsthm,enumerate,color,tikz,fullpage,ulem}
\usepackage[alphabetic]{amsrefs}

\def \bs {\backslash}	
\def \C{{\mathbb C}}
\def \CB{\mathcal{B}}

\def \Ga{\Gamma}
\def \ga{\gamma}
\def \GL{\operatorname{GL}}
\def \Hom{\operatorname{Hom}}

\def \Ind{\operatorname{Ind}}
\def \la{\lambda}

\def \PGL{\operatorname{PGL}}
\def \PGSP{\operatorname{PGSp}}
\def \GSp{\operatorname{GSp}}
\def \Spin{\operatorname{Spin}}

\def \Q{{\mathbb Q}}
\def \R{{\mathbb R}}
\def \SL{\operatorname{SL}}
\def \Sp{\operatorname{Sp}}

\def \tr{\operatorname{tr}}

\def \Z{{\mathbb Z}}
\def \({\left(}
\def \){\right)}
\def \deg {\operatorname{deg}}
\def \G {\mathbb{G}}
\def \wt {\mathrm{wt}}
\def \A {\mathcal{A}}
\def \ord {\mathrm{ord}}
\def \n {n_\Ga}
\def \diag{\mathrm{diag}}
\def \P{\mathcal{P}}
\def \spin {\mathrm{spin}}
\def \st {\mathrm{st}}
\def \semi {\mathrm{semi}}
\def \ext {\mathrm{ext}}
\def\B{\mathcal{B}}
\def\O{\mathcal{O}}

\def \La {\Lambda}
\newtheorem{thm}{Theorem}[section]
\newtheorem{theorem}[thm]{Theorem}

\newtheorem{proposition}[thm]{Proposition}
\numberwithin{equation}{section} \theoremstyle{definition}

\newcommand{\St}{\rm{St}}
\newcommand{\triv}{{\mathbf1}}

\begin{document}

\title{Zeta and L-functions of finite quotients of apartments and buildings}
\author{Ming-Hsuan Kang, Wen-Ching Winnie Li, and Chian-Jen Wang}
\address{Ming-Hsuan Kang\\ Department of Applied Mathematics\\National Chiao-Tung University\\
Hsinchu, Taiwan} \email{\tt mhkang@nctu.edu.tw}
\address{Wen-Ching Winnie Li\\ Department of Mathematics\\ The Pennsylvania State University\\
University Park, PA 16802, U.S.A.} \email{\tt wli@math.psu.edu}
\address{Chian-Jen Wang \\  Department of Mathematics\\ Tamkang University \\ New Taipei City , Taiwan} \email{\tt 142888@mail.tku.edu.tw}
\thanks{The research of the first author is supported by the NSC grant 103-2115-M-009 -006.
The research of the second author is partially supported by the NSF grant DMS-1101368 and the Simons Foundation grant \# 355798. The research of the third author is supported by the NSC grant 103-2115-M-032-001. Part of the research was carried out when the authors visited the National Center for Theoretical Sciences in Hsinchu, Taiwan. They would like to thank  NCTS for its hospitality and support. }


\begin{abstract} In this paper,
we study relations between Langlands $L$-functions and zeta functions of geodesic walks and galleries for finite quotients of the apartments of $G=$PGL$_3$ and PGSp$_4$ over a nonarchimedean local field with $q$ elements in its residue field. They give rise to an identity (Theorem 5.3) which can be regarded as a generalization of Ihara's theorem for finite quotients of the Bruhat-Tits trees. This identity is shown to agree with the $q=1$ version of the analogous identities for finite quotients of the building of $G$  established in \cite{KL1, KLW, FLW}, verifying the philosophy of the field with one element by Tits.
A new identity for finite quotients of the building of PGSp$_4$ involving the standard $L$-function (Theorem 6.3), complementing the one in \cite{FLW} which involves the spin $L$-function, is also obtained.
\end{abstract}

\date{}
\maketitle
\def \aff {\mathrm{aff}}
\def \k {k_\Ga}
\section{Introduction}

For a discrete subgroup $\Ga$ of an algebraic group $G$ over a nonarchimedean local field $F$ with $q$ elements in its residue field, two topics play important roles in the study of $\Ga$. The first is representation theoretic. The unramified irreducible subrepresentations of the regular representation of $G$ on $L^2(\Ga \bs G)$ admit Langlands $L$-functions
defined using their Satake parameters. The second is geometric in nature.
 The classifying space of $\Ga$ can be chosen as the quotient by $\Ga$ of the Bruhat-Tits building $\CB$ of $G$, denoted by $\CB_\Ga$.
The building $\CB$ is the union of apartments, which are Euclidean spaces endowed with the structure of a simplicial complex.
These special features allow us to define geodesics and their counting functions, called zeta functions.

The connection between the $L$- and zeta functions was first studied by Ihara \cite{Ih} in 1960's for the case $G=\PGL_2$, a split algebraic group of rank one.
{Essentially}, he showed that the zeta function (of geodesic walks) and the $L$-function of a discrete cocompact torsion-free subgroup $\Ga$ of $G$ are both rational functions in $u$ and their ratio is equal to $(1-u^2)$ raised to the Euler characteristic of the finite graph $\CB_\Ga$. Bass \cite{Ba} generalized Ihara's result from cocompact to cofinite subgroups by adding a weight function in zeta functions;
Hashimoto \cite{Ha1} studied the subgroups of $SU(3)$, which is a non-split algebraic group of rank one.

Since the space $L^2(\Ga \bs G)$ and the building of $G$ are well-defined for higher rank algebraic groups,
it is natural to seek extensions of Ihara's theorem to discrete subgroups of such groups.  Little was known until the recent work \cites{KL1,KLW} and \cite{FLW}, which generalize Ihara's result to certain split algebraic groups of rank two, namely $\PGL_3$ and $\PGSP_4$, respectively. For a rank two group, in addition to the zeta function of geodesic walks as in a rank one group, one also considers the zeta function of geodesic galleries. These are generalizations of the Selberg zeta function to the $p$-adic setting. The extended Ihara's theorem for rank two groups is an identity involving the $L$-function and these two kinds of zeta functions, reminiscent of the zeta function for a smooth irreducible projective surface defined over a finite field.  Moreover, the Artin $L$-functions of the discrete torsion-free cocompact subgroups of PGL$_3$,  extending those for PGL$_2$ in \cites{Ih, Ha2, Ha3}, are investigated in \cite{KL2}, where a similar identity involving the Artin $L$-function is established.

For algebraic groups of higher rank, as revealed in \cites{KL1, KLW, FLW}, the situation is a lot more complicated. It could be illuminating to first consider the degenerate case following the philosophy of the field with one element introduced by Tits.
This was considered for $G=$PGL$_n$ over the 1-adic field $\Q_1$ in \cite{DK}. In this case, the group $G(\Q_1)$ is isomorphic to the extended affine Weyl group and its building is a single apartment. Other approaches to study zeta functions for algebraic groups of higher rank are explored in \cite{KM} and \cite{DKM}, which only concern geodesics in the highest dimension. 

The purpose of this paper is to study $L$- and zeta functions {associated to} finite index subgroups $\Ga$ of the extended affine Weyl group $W_\ext$ of a simple split algebraic group $G$ {of adjoint type. In this case,
$W_\ext$ modulo the Weyl group can be identified with the hyper-special vertices} of the standard apartment $\A$ of the building $\CB$ attached to $G(F)$; thus this is indeed the degenerate case of our ultimate goal. We show how these functions are related. Compared to the aforementioned known results \cites{KL1, KLW, FLW} in the building case, the relations in the apartment case appear to be more transparent. The $L$- and zeta functions are attached to an irreducible representation of the Langlands dual group $\hat G(\mathbb C)$ of $G(F)$, which makes relations much clearer and also allows for further generalization.
Special focus is on the groups $G =$PGL$_3$ and PGSp$_4$. We develop a unified approach to obtain an identity  extending Ihara's theorem for such $G$ in the apartment case (Theorem \ref{Aidentity}), and compare them with the results established in \cites{KL1, KLW, FLW} for the building of $G$.

This paper is organized as follows. The basic concepts of an apartment $\A$ and Langlands $L$-functions are reviewed in \S 2, and zeta functions of walks are defined in \S3. In \S5 a new concept, the gallery of $\pi$-chambers, and the zeta function of such galleries are introduced. Theorem \ref{infinitesum2} relates the $L$-function to counting closed walks in $\A$. No similar expressions for the building $\B$ are known. The quotient $\A_\Ga = \Ga \backslash \A$ is a simplicial torus when $\Ga$ is a subgroup of the coroot lattice. This case is handled in \S3. \S4 and \S5 are devoted to understanding the remaining case, where $\A_\Ga$ is a Klein bottle provided that $G$ has rank $2$. For this we specialize to $G =$PGL$_3$ and PGSp$_4$. The $L$-functions and zeta functions of walks are studied in \S4 and zeta functions of galleries in \S5.
 Explicit connections between the $L$-function and the zeta function of walks are given in Theorem \ref{maintheorem0} for tori and Theorem \ref{maintheorem1} for Klein bottles; and the relation between the zeta function of walks and that of galleries is summarized in Theorem \ref{zeta2andzeta} for both cases. Combined, they yield Theorem \ref{Aidentity}, a generalization of Ihara's theorem. In \S6 we turn to the buildings of $G=$PGL$_3$ and PGSp$_4$.
It is shown that the zeta functions studied in \cites{KL1, KLW, FLW} agree with the zeta functions defined in this paper using group theoretic language, and, upon setting $q=1$, the (more complicated) identities established in  \cites{KL1, KLW, FLW} for buildings do reduce to the identities for apartments, verifying the philosophy of Tits. Moreover, we obtain a new identity, Theorem \ref{st-building}, for finite quotients of the building of PGSp$_4$ involving the $L$-function associated to the degree-$5$ standard representation of Spin$_5(\mathbb C)$. This complements the results in \cite{FLW} for the $L$-function associated to the degree-$4$ spin representation of Spin$_5(\mathbb C)$.  By letting $q=1$ in the new identity, we see that it also compares well with the corresponding identity on the apartment.

\bigskip

\section{Preliminary}
\subsection{Root data and dual groups}
Let $G$ be a connected  simple split algebraic group over a local field $F$ with a discrete valuation $\ord_F$. Let $T$ be a maximal split torus contained in a Borel subgroup $B$.
The rational character group and the cocharacter group of $T$ are
$$ X(T) = \Hom(T,\G_m) \qquad \mathrm{and} \qquad Y(T) =\Hom( \G_m,T) $$
respectively, which are free $\Z$-modules of rank equal to $\dim T$.
There is a perfect paring $\langle , \rangle $ between $X(T) $ and $Y(T)$ characterized by
$$ \la(\rho(a)) = a^{\langle \la, \rho \rangle}$$
for all $\la \in X(T), \rho \in Y(T)$ and $a \in F$.
Let $\Phi$ be the set of nontrivial characters of $T$ in the adjoint representation of $G$, called the set of roots of $G$.
The roots contained in $\mathrm{Lie}(B)$ are positive roots $\Phi^+$ which contains a set of simple roots $\Delta$ as a $\Z$-basis of $X(T)$.
Let $\check{\Phi} = \{ \check{\alpha}, \alpha \in \Phi\} \subset Y(T)$ be the set of coroots of $G$ characterized by the following properties:
\begin{enumerate}
\item $\langle \alpha, \check{\alpha} \rangle = 2.$
\item $s_\alpha(x) := x - \langle x, \check{\alpha} \rangle \alpha$ defines an involution on $\Phi$.
\end{enumerate}
Furthermore, for $\check{\alpha} \in \check{\Phi}$, $s_{\check{\alpha}}(x) = x- \langle \alpha,x  \rangle \check{\alpha}$ also defines an involution on $\check{\Phi}$.
The quadruple $(X(T), \Phi, Y(T), \check{\Phi}$) is called the root datum of $G$.
There is a connected simple split algebraic group, unique up to isomorphism, with the dual root datum  $( Y(T), \check{\Phi}, X(T), \Phi$), called the dual group of $G$.

We are interested in two dual groups of $G$: the dual group $\hat{G}(F)$ over the same based field $F$ and the complex dual group $\hat{G}(\C)$.
There is a geometric object associated to $\hat{G}(F)$ known as the Bruhat-Tits building on which $G(F)$ acts as automorphisms \cite{Ti}.
On the other hand, the complex dual group $\hat{G}(\C)$ parametrizes part of the irreducible complex representations of $G$. We investigate the interplay between these two objects.

\subsection{The apartment associated to $T$}

Consider $V=Y(T) \otimes \R$ endowed with a positive definite bilinear form $(, )$ preserved by the actions of $\{s_{\check{\alpha}} , \check{\alpha} \in \check{\Phi}\}$.
In this case, for $\alpha \in \Phi$ and $\beta \in \check{\Phi}$, we have
$$ \langle \alpha, \beta \rangle = 2\frac{ ( \check{\alpha}, \beta )}{ ( \beta, \beta)}.$$
The group generated by  $\{s_{\check{\alpha} } , \check{\alpha}  \in \check{\Phi}\}$ is the Weyl group $W$ of $G$, whose elements are linear isometries on $V$.

By abuse of notation, we identify $v\in V$ as the translation on $V$ mapping $x$ to $x+v$. Then the extended affine Weyl group is a subgroup of isometries on $V$ generated by the translation subgroup $Y(T)$ and the Weyl group $W$, expressed as a semi-direct product
$$ W_{\ext} =  Y(T) \rtimes W.$$

Except for $Y(T)$, there are two important lattices in $V$: the coroot lattice $\La_r$ spanned by coroots,  and the coweight lattice $\La$ given by
$$ \Lambda = \{ \lambda \in V, \langle \alpha, \la \rangle \in \Z , \forall \alpha \in \Phi \}.$$
In general,
$$  \Lambda \supseteq Y(T) \supseteq \Lambda_r.$$
Further, $\Lambda = Y(T)$ when $G$ is of adjoint type,  and  $Y(T)=\Lambda_r $
when $G$ is simply connected.

The affine Weyl group is a subgroup of $W_{\ext}$ generated by the coroot lattice and the Weyl group, given by
$$ W_\aff= \Lambda_r \rtimes W.$$
A reflection in $W_{\mathrm{aff}}$ is of the form $(k \check{\alpha} , s_{\check{\alpha} })$ for some $k \in \Z$ and $\check{\alpha}  \in \check{\Phi}$; it fixes the hyperplane
$$ H_{\check{\alpha} ,k} = \{ v \in V, \langle \alpha , v \rangle= k \}$$
called a wall. The apartment associated to $T$, denoted by $\A$, is a simplicial complex with $V$ as its underlying space. The simplices of highest dimension, called chambers, are the closure of the connected components of $V$ with the walls removed. Lower dimensional simplices are intersections of chambers.
The vertices of $\A$ in the coweight lattice $\Lambda$ are called hyper-special vertices, whose stabilizers in automorphisms of $\A$ are isomorphic to $W$.

Note that when $G$ is of adjoint type, $Y(T)=\Lambda$ and the extended affine Weyl group $W_\ext$ acts transitively on hyper-special vertices.
Moreover, one can assign  $(\dim(V)+1)$ types to the vertices such that the vertices of each chamber have distinct types. Then $W_{\mathrm{aff}}$ preserves types and it acts transitively on hyper-special vertices of a fixed type.

Let $T_0$ be the maximal compact subgroup of $T$. There is a $\Z$-module isomorphism between $T/T_0$ and $Y(T)$  such that the $T_0$-coset represented by $t = (t_1,\cdots,t_n)$ is mapped to $\rho_t \in Y(T)$  given by
$$\rho_t(x) = (x^{\ord_F(t_1)},\cdots ,x^{\ord_F(t_n)}).$$
In this case, we can identify the extended affine Weyl group as
$$W_\ext \cong N(T)/T_0 \cong T/T_0 \rtimes N(T)/T$$
where $N(T)$ is the normalizer of $T$ in $G$.

\subsection{Apartment as weight space}
Let $\hat{G}(\C)$ be the complex dual group of $G$ with the maximal torus $\hat{T}(\C)$ such that $X(\hat{T})$ is identified with $Y(T)$ and its roots and positive roots are $\hat{\Phi}$ and $\hat{\Phi}^+.$
In this case, we can regard the space $V = Y(T)\otimes_{\Z} \R$ as the weight space of $\hat{G}(\C).$

Recall that for the set of simple roots $\Phi^+$, there is a set of weights $\{f_\alpha, \alpha \in \Phi^+\}$ in $X(\hat{T})$ characterized by
$$ \langle \beta, f_\alpha \rangle = \delta_{\alpha, \beta} \quad \mbox{ for all }\alpha, \beta \in \Phi^+$$
called the set of fundamental weights of $\hat{G}(\C)$.

A complex representation $\pi$ of  $\hat{G}(\C)$ is called minuscule (resp. quasi-minuscule) if its weights (resp. nontrivial weights) form a single $W$-orbit of some fundamental weight. Denote by $\wt(\pi)$ the set of weights of $\pi$, and by $\wt'(\pi)$ the subset of nontrivial weights in $\wt(\pi)$.

\subsection{Langlands L-functions}
By the work of Casselman \cite{Ca}, every irreducible unramified representation $\rho$ of {$G$}  occurs in some unramified principal series representation $I_\chi$, where $\chi$ is a character of $T$ trivial on $T_0$. Such $\chi$ is uniquely determined up to the action of the Weyl group $W$.
On the other hand, we have a canonical W-equivariant isomorphism among the groups
$$ \Hom(T/T_0,\C^\times) \cong \Hom(Y(T),\C^\times) \cong \Hom(X(\hat{T}),\C^\times) \cong \hat{T}(\C).$$
Thus, there is a $W$-orbit in the complex torus $\hat{T}(\C)$ corresponding to the unramified character $\chi$, called the Satake parameter of $\rho$, denoted by $s_\rho$ \cite{Sa}.
Fix a finite-dimensional representation $\pi$ of $\hat{G}(\C)$. Then the local
{Langlands} $L$-function attached to $(\rho,\pi)$ is given as
\begin{align*}
L(\rho, \pi, u)
&= \det(1- \pi( s_\rho)u)^{-1} ,
\end{align*}
which is independent of the choice of $s_\rho$. Moreover, let $\wt(\pi)$ be the set of weights of $\pi$ and regard $\chi$ as a character on $X(\hat{T})$, then
 \begin{align*}
L(\rho,\pi, u) = \prod_{\la \in \wt(\pi)} (1- \chi(\la) u )^{-1}.
\end{align*}

\def \Res {\mathrm{Res}}

\subsection{L-functions for extended affine Weyl groups}
Before defining the L-functions for extended affine Weyl groups, we recall representations of such groups.
As the Weyl group $W$ is the maximal finite (and hence compact) subgroup of $W_\ext$, an irreducible representation of $W_\ext$ is called unramified if it has a non-zero $W$-fixed vector.
For a finite-dimensional irreducible representation $\tau$ of $W_\ext$,  let $\chi$ be a character of $Y(T) \cong T/T_0$ occurring in the restriction representation $\Res^{W_\ext}_{Y(T)} \tau$ of $Y(T)$. By Frobenius reciprocity,
$$ \Hom_{W_\ext}(\tau, \Ind^{W_\ext}_{Y(T)} \chi) = \Hom_{Y(T)}(\Res^{W_\ext}_{Y(T)} \tau, \chi) \neq \{0\}.$$
Note that the $W$-fixed subspace of $\Ind_{Y(T)}^{W_\ext} \chi$ is one-dimensional.
Therefore, an unramified representation $\tau$ of $W_\ext$ is isomorphic to the unique unramified sub-representation of some $\Ind_{Y(T)}^{W_\ext} \chi$.
By the classification in \cite{Se} of irreducible representations of semidirect products by an abelian group,  such $\chi$ is unique up to the action of $W$. The $L$-function for an unramified representation $\tau$ of $W_{\ext}$ associated to the representation $\pi$ of $\hat{G}(\C)$ is defined as
\begin{align*}
L_{\ext}(\tau,\pi,  u) = \prod_{\la \in \wt(\pi)} (1- \chi(\la) u )^{-1},
\end{align*}
which coincides with the $L$-function of the unique irreducible unramified representation $\rho$ of $G$ occurring in the principal series $I_\chi$.

\section{Zeta functions of geodesic walks and L-functions}
From now on, assume that $G$ is of adjoint type so that the coweight lattice $\Lambda$ is equal to  the co-character group $Y(T) \cong W_\ext / W$.
Fix a torsion-free subgroup $\Ga$ of $W_{\mathrm{aff}}$ such that $\Ga$ preserves types of vertices. Then $\A_\Ga=\Ga \bs \A$ is a finite complex with the universal cover $\A$ and the fundamental group $\Ga$.
Denote by $\Ga\bs \Lambda$ the set of hyper-special vertices on $\A_\Ga$. We shall fix a set of representatives of $\Ga \bs \Lambda$ in $\Lambda$ and still denote it by $\Ga \bs \Lambda$.
Let $\pi$ be an irreducible representation of $\hat{G}(\C)$ whose nontrivial weights form the $W$-orbit of a fundamental weight. In other words, $\pi$ is either a minuscule representation or a fundamental quasi-minuscule representation.
For example,
\begin{itemize}
\item For $\hat{G}(\C)=\SL_n(\C)$, $\pi$ is an exterior power of the standard representation.
\item For $\hat{G}(\C)=\Spin_{2n+1}(\C)$, $\pi$ is the spin representation of dimension $2^n$ or the standard representation of dimension $2n+1$.
\end{itemize}

We shall identify the weights in $\wt(\pi)$  of $\pi$ as a subset of $\Lambda$.

\subsection{Rational geodesics in an apartment}
A geodesic $L$ in $\A \cong \R^n$ is a straight line with a fixed orientation, called its direction. Given a subset $S$ of $\Lambda$, a line $L$ is called an $S$-geodesic if it
 has the same direction as some element in $S$.
When $S$ is the set of nontrivial weights of $\pi$, we also call $L$ a $\pi$-geodesic.
Furthermore, $L$ is called a rational geodesic if it is contained in the 1-skeleton of $\A$.
\subsection{Rational geodesics in a finite quotient}

Let $C_{\pi}$ be the collection of closed walks contained in the 1-skeleton of $\A_\Ga$ which lift to a part of a $\pi$-geodesic in $\A$. Here a closed walk has a starting point. Two closed walks $c$ and $c'$ in $\A_\Ga$ are called equivalent if one can be obtained from the other by switching the starting vertex. A closed walk $c$ in $C_\pi$ is called a {\it geodesic walk} if all walks equivalent to $c$ are in $C_\pi$. Thus whether an element $c$ in $C_{\pi}$ is a geodesic walk depends on how it closes on itself. More precisely, a lift of $c$ in $\A$ is the line segment from $v$ to $g(v)$ for some $g \in \Gamma$.
Then the line segment from $g(v)$ to $g^2(v)$ is also a lift of $c$, and $c$ is a geodesic walk if and only if the direction from $v$ to $g(v)$ agrees with that from $g(v)$ to $g^2(v)$. In other words, $c$ is a geodesic walk if there is no corner at $g(v)$, or equivalently, $c$ repeated twice also lies in $C_\pi$. The following picture depicts a lifting of a closed walk that is not a geodesic walk, in which the element $g \in \Gamma$ is a glide reflection explained in \S4.1.

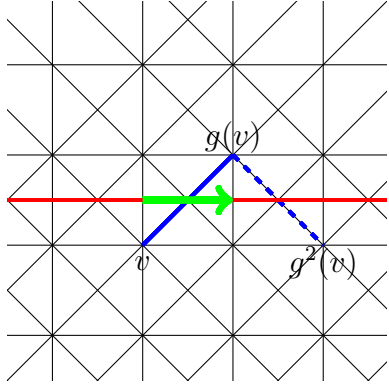
\begin{center}
\begin{tikzpicture}[scale=1.2]
\clip (-.5,-.5) rectangle ++ (4.2,4.2);

\foreach \y in {-2,-1,...,5}
{
\draw plot (\x,\y);
\draw plot (\y,\x);
\draw plot (\x+\y,\x);
\draw plot (\x+\y,-\x);
\draw [line width=.5mm, blue,]  (1,1)--(2,2);
\draw [line width=.5mm, blue,dashed]  (2,2)--(3,1);
\draw [line width=.5mm, red]  (-2,1.5)--(5,1.5);
\draw [line width=1mm, green,->]  (1,1.5)--(2,1.5);
}
\path
node (v1) at (1,.8) [] {$v$}
node (v1) at (2,2.2) [] {$g(v)$}
node (v1) at (3,.8) [] {$g^2(v)$}
;
\end{tikzpicture}
\end{center}
%\kang{looks locally like a walk, but there may be an angle when it closes on itself, which case it is not a geodesic %walk.} 
Recall that all nontrivial weights in $\wt(\pi)$ have the same length $\ell(\pi)$.
 Thus the length of a closed walk $c$ in $C_\pi$ is an integral multiple of $\ell(\pi)$. This integral multiple is called the normalized length of $c$, denoted by $l(c)$.
Denote by $N_n$ (resp. $ \tilde{N}_n$) the number of closed walks (resp. closed geodesic walks) in $C_\pi$ of normalized length $n$. A closed walk is called {\it primitive} if it is not a repetition of a shorter closed walk.

%\begin{center}
%\begin{tikzpicture}[scale=1.2]
%\clip (-.5,-.5) rectangle ++ (4.2,4.2);

%\foreach \y in {-2,-1,...,5}
%{
%\draw plot (\x,\y);
%\draw plot (\y,\x);
%\draw plot (\x+\y,\x);
%\draw plot (\x+\y,-\x);
%\draw [line width=.5mm, blue,]  (1,1)--(2,2);
%\draw [line width=.5mm, blue,dashed]  (2,2)--(3,1);
%\draw [line width=.5mm, red]  (-2,1.5)--(5,1.5);
%\draw [line width=1mm, green,->]  (1,1.5)--(2,1.5);
%}
%\path
%node (v1) at (1,.8) [] {$v$}
%node (v1) at (2,2.2) [] {$g(v)$}
%node (v1) at (3,.8) [] {$g^2(v)$}
%;
%\end{tikzpicture}
%\\
%\kang{For the glide reflection $g \in \Gamma $, $v \to g(v)$ does not project to a closed geodesic walk in the %quotient $\A_\Gamma$ 
%since there is an angle when the walk closes on itself under projection.}
%\end{center}

\subsection{Zeta functions of geodesic walks}
The zeta function of $\pi$-geodesic walks on $\A_\Ga$ is defined as
$$ Z(\A_\Ga,\pi,u) = \prod_{c}(1- u^{l(c)})^{-1}, $$
where $c$ runs through all equivalence classes of primitive rational closed $\pi$-geodesics in $\A_\Ga$,
 %that is, all primitive closed geodesics in $\A_\Ga$ of type $\pi$.
As such, the zeta function of geodesic walks may be regarded as an analogue of the Selberg zeta function.

It should be pointed out that unlike the case of the full building where the similarly defined zeta function is an infinite product converging to a rational function, the above zeta function on an apartment is a finite product and the reciprocal of the zeta function is a polynomial.
A straight forward computation similar to the case of graphs gives the following theorem.
\begin{theorem} \label{infinitesum1}
For $|u|\ll 1$, the following identity holds
$$ Z(\A_\Ga,\pi,u) = \exp\left(\sum_{n=1}^\infty \frac{\tilde{N}_n}{n}u^{n}\right).$$
\end{theorem}

\def \Irr {\mathrm{Irr}}
\def \unr {\mathrm{unr}}
\subsection{ L-functions of type $\pi$ }
Denote by $\Irr^\unr(\Ga\backslash W_\ext)$ the set of irreducible unramified representations of $W_\ext$ in  $L^2(\Ga \bs W_\ext)$ (counting the multiplicities).
{Define the $L$-function of type $\pi$ associated to $\Ga\bs W_\ext$ as}
$$ L(\Ga\bs W_\ext, \pi, u) := \prod_{\rho \in \Irr^\unr(\Ga\backslash W_\ext)} L_\ext(\rho, \pi, u)= \prod_{\rho \in \Irr^\unr(\Ga \backslash W_\ext)} \prod_{\la \in \wt(\pi)} (1- \chi_\rho( \la) u )^{-1}.$$
Like zeta functions of geodesic walks, the $L$-function $L(\Ga\bs W_\ext ,\pi, u)$ also can be expressed in terms of closed walks as follows.
\begin{theorem} \label{infinitesum2}
When $|u|<1$, the following identity holds:
{
$$ L(\Ga \bs W_\ext, \pi,  u) =(1-u)^{-\epsilon(\pi) N}  \exp\left( \sum_{n=1}^\infty \frac{N_n}{n} u^{n} \right).$$
Here $\epsilon(\pi) = |\wt(\pi) \smallsetminus \wt'(\pi)|$ is the multiplicity of the trivial weight in $\wt(\pi)$ and $N$ is the cardinality of $\Ga \bs W_{\ext}/W$.}
\end{theorem}
\begin{proof}
Recall that
\begin{align*}
N_n &= \mbox{the number of {rational} closed walks in $\A_\Ga$ of normalized length $n$ which lift to } \\
&\quad \mbox{ a part of a $\lambda$-geodesic in $\A$ for some $\la \in \wt'(\pi)$} \\
&= \sum_{\la \in \wt'(\pi)} \sum_{\gamma \in \Ga} \# \{ x \in \Ga\bs \Lambda : \ga(x)   = x+ n\la\}.
\end{align*}

Fix a weight $\tilde{\lambda}$ in $\wt'(\pi)$. Consider the action of the Hecke operator $B_n = W\tilde{\lambda}^n W = \cup_{\lambda \in \wt'(\pi)} \lambda^n W$   on $L^2( W_\ext /W)=L^2( \Lambda)$ which
 sends $f(x) \in L^2( \Lambda)$ to
$$ B_n f(x)= \sum_ {\lambda \in \wt'(\pi)}f(x + n\lambda).$$
Let $\rho$ be an unramified irreducible representation of $W_\ext$. By definition it occurs in the induced representation $\Ind_{\Lambda}^{W_\ext} \chi_\rho$ for some character $\chi_\rho$ of $\Lambda$. Then $B_n$ sends a $W$-fixed vector $v$ of $\rho$ to
$$ B_n(v)= \sum_ {\la \in \wt'(\pi)}\chi_\rho(n \lambda )v.$$
Since the space of $W$-invariant vectors for $\rho$ is $1$-dimensional,
 this shows that {$\Irr^\unr(\Ga \backslash W_\ext)$ and $\Ga \bs W_\ext / W$ have the same cardinality $N$} and the trace of $B_n$ on  $L^2(\Ga \bs W_\ext /W) = L^2(\Ga \bs \Lambda)$ is
$$ \tr(B_n) =\sum_{\rho \in \Irr^\unr(\Ga \backslash W_\ext)} \sum_ {\la \in \wt'(\pi)}\chi_\rho(\lambda )^n.$$
On the other hand, by choosing the characteristic functions on the $\Ga$-orbits of $\Lambda$ as a basis of  $L^2(\Ga \bs \Lambda)$, we can interpret the trace of $B_n$ as
\begin{align*}
\tr(B_n) &= \sum_{x \in \Ga \bs \Lambda} \# \{\la \in \wt'(\pi): \Ga(x) = \Ga(x+n\lambda)\}. \\
\end{align*}
Note that if $x \in \Lambda$ is such that $\ga_1 (x)= \ga_2 (x)$ for some $\ga_1,\ga_2 \in \Ga$, then we have $ \ga_1^{-1}\ga_2 \in \mathrm{Stab}(x)\cap \Ga = \{\rm{id}\}$ since $\Ga$ is torsion-free.
Therefore, when $\Ga(x) = \Ga(x+n\lambda)$, there exists a unique $\ga \in \Ga$ such that $\ga(x)=x+n \lambda$. Hence we conclude
\begin{align*}
\tr(B_n)&= \sum_{x \in \Ga \bs \Lambda} \# \{ (\lambda,\ga) \in \wt'(\pi)\times \Ga ~:~ \ga(x)  = x+ n \lambda\}= N_n.
\end{align*}
Consequently, by letting $\epsilon(\pi) = |\wt(\pi) \smallsetminus \wt'(\pi)|$, we have
{
\begin{align*}
\sum_{n=1}^\infty N_n \frac{u^n}{n}
&= \sum_{n=1}^\infty ~\sum_{\rho \in \Irr^\unr(\Ga \backslash W_\ext)}~ \sum_{\la \in \wt'(\pi)} \chi_\rho( \la)^n \frac{u^n}{n} \\
&= \sum_{\rho \in \Irr^\unr(\Ga \backslash W_\ext)} ~\sum_{\la \in \wt'(\pi)} \log(1-\chi_\rho( \la) u)^{-1}\\
&=\log\left(\prod_{\rho \in \Irr^\unr(\Ga \backslash W_\ext)} ~\prod_{\la \in \wt'(\pi)} (1-\chi_\rho( \la) u)\right)^{-1}  \\
&=\log(1-u)^{N\epsilon(\pi)} { + } \log \left(\prod_{\rho \in \Irr^\unr(\Ga \backslash W_\ext)} ~\prod_{\la \in \wt(\pi)} (1-\chi_\rho( \la) u)\right)^{-1}\\
&= \log(1-u)^{N\epsilon(\pi)} { + } \log L(\Ga \bs W_\ext,\pi, u),
\end{align*}
}
which completes the proof.
\end{proof}

\subsection{Zeta and L-functions for simplicial tori}

Suppose $\Ga$ is a subgroup of $\Lambda_r$ and $[\Lambda:\Ga]=N$, then  $\A_\Ga$ is a simplicial torus with $N$ vertices. {Let $\widehat{\Lambda / \Ga}$ be the set of characters of $\Lambda$ trivial on $\Ga$.
Then
$$ L^2(\Ga \bs W_\ext) = \bigoplus_{\chi \in \widehat{\Lambda / \Ga}} \Ind_{\Lambda}^{W_{\ext}} \chi.$$
%\kang{
%Let $C(W_\ext)$ be $\C$-valued functions on $W_\ext$, which induced a representation %of $W_\ext$ by left multiplication.
%Let $C(\Ga \backslash \Lambda)$ be $\C$-valued functions on $\Lambda$ which is %invariant under the action of $\Ga$. As a representation of $\Lambda$,
%$$ C(\Ga \backslash \Lambda) = \bigoplus_{\chi \in \widehat{\Lambda / \Ga}} \chi.$$
%Hence, as a representation of $W_\ext$,
%$$ C( \Ga \backslash W_\ext) = \Ind^{W_{\ext}}_\Lambda C(\Ga \backslash \Lambda) =
%\bigoplus_{\chi \in \widehat{\Lambda / \Ga}} \Ind_{\Lambda}^{W_{\ext}} \chi.$$}

\noindent Therefore,
\begin{align*}
 L(\Ga\bs W_\ext,\pi , u)
 &= \prod_{\chi \in \widehat{\Lambda / \Ga}}\prod_{\la \in \wt(\pi)} (1- \chi( \la) u )^{-1}=\prod_{\la \in \wt(\pi)}   \left( \prod_{\chi \in \widehat{\Lambda / \Ga}}(1- \chi(\la) u)^{-1} \right).
\end{align*}

Let $\rho_{\mathrm{reg}}$ be the regular representation of $\Lambda / \Ga$, then
$$\prod_{\chi \in \widehat{\Lambda / \Ga}}(1- \chi(\la) u) =\det(1- \rho_{\mathrm{reg}}(\la) u)$$
and
\begin{align*}
(1-u)^{\epsilon(\pi) N}  L(\Ga \bs W_\ext,\pi , u) &=\prod_{\la \in \wt'(\pi)}  \det(1- \rho_{\mathrm{reg}}(\la) u)^{-1}\\
&=\prod_{\la \in \wt'(\pi)}   (1- u^{\deg(\la)})^{-N/\deg(\la)}.
\end{align*}
Here $\deg(\la)$ is the order of $\la$ in the quotient group  $\Lambda / \Ga$.
On the other hand, since every closed walk is closed by a translation, every element in $C_\pi$ has no corner at the starting vertex.
Therefore,
$$ N_n= \tilde{N}_n  \quad \mbox{ for all $n$.}$$
Together with Theorems \ref{infinitesum1} and \ref{infinitesum2}, we summarize the above results as
\begin{theorem} \label{maintheorem0}
Suppose $\Ga$ is a finite index subgroup of $\Lambda_r$ with $N = [\Lambda : \Ga]$ and $\pi$ is a {minuscule  or  fundamental quasi-minuscule representation} of $\hat{G}(\C)$. Then
$$ Z(\A_\Ga,\pi,u)=(1-u)^{\epsilon(\pi) N} L(\Ga \bs W_\ext, \pi , u) =  \prod_{\la \in \wt'(\pi)}  (1- u^{\deg(\la)})^{-N/\deg(\la)}. $$
Here $\epsilon(\pi)$ is the multiplicity of the trivial weight in $\wt(\pi)$.
\end{theorem}

{When $G=\PGL_n(F)$ and $\pi$ is the standard representation of $\SL_n(\C)$, the above result agrees with  Theorem 5.2 in \cite{DK}.}

\section{Zeta functions of geodesic walks and L-functions for Klein bottles}
Next assume $\Ga$  is not contained in $\Lambda_r$. Then the situation is more complicated.
We shall assume that the apartment of $G$ is two-dimensional.
In this case, $\A_\Ga$ is a Klein bottle.
Recall that there are three types of two-dimensional apartments: $\tilde{A}_2, \tilde{B_2} = \tilde{C_2},$ or $\tilde{G_2}$.
In the remaining discussion, we only consider $\tilde{A}_2$ and $\tilde{C_2}$ cases.

For type  $\tilde{A}_2$, we shall be concerned with two minuscule representations $\pi_1$ and $\pi_2$ of $\hat{G}(\C)$ with weights as follows:

\begin{center}
\begin{tikzpicture}[scale =.75]
\clip (-4.2,-4.2) rectangle ++ (4.6,4.6);
\foreach \y in {-10,-9,...,10}
{
\draw plot (\x+\y,1.732*\x);
\draw plot (-\x+\y,1.732*\x);
\draw plot (-\x,0.866*\y);
}
\draw[red, very thick,->] (-2,-1.732)--(-1,-1.732);
\draw[red, very thick,->] (-2,-1.732)--(-2.5,-2.598);
\draw[red, very thick,->] (-2,-1.732)--(-2.5,-.866);
 \node at (-1,-1.4) {$\alpha$};
\end{tikzpicture}
\hspace{.8cm}
\begin{tikzpicture}[scale =.75]
\clip (-4.2,-4.2) rectangle ++ (4.6,4.6);
\foreach \y in {-10,-9,...,10}
{
\draw plot (\x+\y,1.732*\x);
\draw plot (-\x+\y,1.732*\x);
\draw plot (-\x,0.866*\y);
}
\draw[blue, very thick,->] (-2,-1.732)--(-1.5,-2.598);
\draw[blue, very thick,->] (-2,-1.732)--(-1.5,-.866);
\draw[blue, very thick,->] (-2,-1.732)--(-3,-1.732);
  \node at (-1.5,-.5) {$\beta$};
\end{tikzpicture}\\
the weights of $\pi_1$ \hspace{2cm} the weights of $\pi_2$
\end{center}

\noindent Note that $\wt(\pi_1)$ and $\wt(\pi_2)$ are opposite. Further a weight $\alpha \in \wt(\pi_1)$ together with a weight $\beta  \in \wt(\pi_2)$ not equal to $-\alpha$ form a basis of $\Lambda$.

For type $\tilde{C}_2$, we shall consider
 the spin representation $\pi_\spin$ of $\hat{G}(\C)$ and
 the standard representation $\pi_\st$ of $\hat{G}(\C)$. Their weights are shown below.
 Note that $\pi_\spin$ is minuscule while $\pi_\st$ is quasi-minuscule.

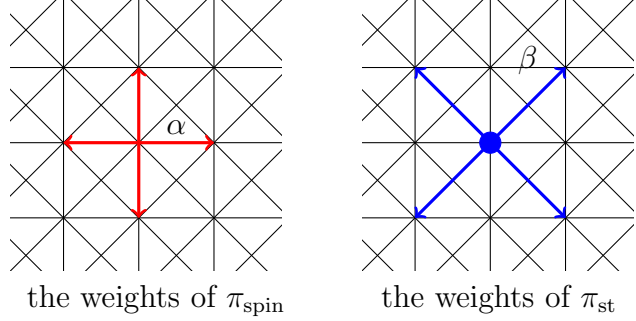
\begin{center}
\begin{tikzpicture}[scale =1]
\clip (-3.7,-3.7) rectangle ++ (3.6,3.6);
\foreach \y in {-10,-9,...,10}
{
\draw plot (\x,\y);
\draw plot (\y,\x);
\draw plot (\x+\y,\x);
\draw plot (-\x+\y,\x);
}
\draw[red, very thick,->] (-2,-2)--(-1,-2);
\draw[red, very thick,->] (-2,-2)--(-3,-2);
\draw[red, very thick,->] (-2,-2)--(-2,-1);
\draw[red, very thick,->] (-2,-2)--(-2,-3);
 \node at (-1.5,-1.8) {$\alpha$};

\end{tikzpicture}
\hspace{.8cm}
\begin{tikzpicture}[scale =1]
\clip (-3.7,-3.7) rectangle ++ (3.6,3.6);
\foreach \y in {-10,-9,...,10}
{
\draw plot (\x,\y);
\draw plot (\y,\x);
\draw plot (\x+\y,\x);
\draw plot (-\x+\y,\x);
}
\draw[blue, very thick,->] (-2,-2)--(-1,-1);
\draw[blue, very thick,->] (-2,-2)--(-3,-3);
\draw[blue, very thick,->] (-2,-2)--(-1,-3);
\draw[blue, very thick,->] (-2,-2)--(-3,-1);
\filldraw [blue] (-2,-2) circle (4pt);
\node at (-1.5,-.9) {$\beta$};
\end{tikzpicture}\\
the weights of $\pi_\spin$ \hspace{1cm} the weights of $\pi_\st$
\end{center}

\noindent In this case a weight $\alpha \in \wt(\pi_\spin)$ together with a nontrivial weight $\beta \in \wt'(\pi_\st)$ form a basis of $\Lambda$.

\subsection{The group structure of $\Ga$}
In this case, $\Ga$ is generated by a translation $t$ and a glide reflection $\sigma$, which is a composition of a linear reflection and a translation, and $\sigma$ does not commute with $t$ \cite{NS}.
Moreover, $t$ and $\sigma^2$ generate the index-$2$ translation subgroup $\Ga_0$ of $\Ga$.

\begin{proposition}
$\sigma^{-1} t \sigma = \sigma^{4k}t^{-1}$ for some integer $k$.
\end{proposition}
\begin{proof}
Since $\sigma^{-1} t \sigma$ is a translation in $\Ga_0$, it is of the form $\sigma^{2m} t^{n}$.
As $\sigma^{-1} t \sigma$ and $\sigma^2$ generate $\sigma^{-1}\Ga_0\sigma = \Ga_0$, we have $n=1$ or $-1$. If $n=1$, then
$$ t = \sigma^{-2} t \sigma^2 = \sigma^{-1}( \sigma^{2m} t) \sigma= \sigma^{4m} t.$$
This implies $m=0$ and hence $\sigma$ commutes with $t$, a contradiction. Therefore $n=-1$. Finally, if $m$ is odd, then one can check directly that $t \sigma^{-m}$ is an element of order two, which is a contradiction. This proves that $m$ is even, as asserted.
\end{proof}
Replace the generator $t$ by  $t\sigma^{-2k}$ such that the relation becomes $t \sigma = \sigma t^{-1}$.
In this case, the direction of the translation $t$ is perpendicular to the reflection axis of $\sigma$. We summarize the above discussion in the following theorem.
\begin{theorem}
If $\A_\Ga$ is a Klein bottle, then $\Ga\cong \langle \sigma, t | t\sigma = \sigma t^{-1}\rangle.$
Consequently, a fundamental domain of $\A_\Ga$ can be chosen to be a rectangle.
\end{theorem}

From now on, we shall fix a choice of $t$ and $\sigma$ in $\Gamma$ satisfying the above theorem.

Next, we study the conjugacy class $[\ga]$ of $\ga \in \Gamma$. If $\ga$ is not a translation, then it is a glide reflection and it can be expressed as $t^n \sigma^m$ for some odd $m$.
Observe that if $n$ is even, then $t^n \sigma^m$ is conjugate to $\sigma^m$; if $n$ is odd, then $t^n \sigma^m$ is conjugate to $t \sigma^m = (t \sigma)^m $. This is summarized in
\begin{proposition}\label{conjugacyclass}
The conjugacy classes of glide reflections in $\Ga$ are $ [\sigma^m]$ and $[ (t \sigma)^m]$ with $m \in 2\Z+1$.
\end{proposition}

The glide reflection $\sigma$ can be written as  $\sigma(x)= \sigma_0(x)+ a\alpha + b \beta$, where $\sigma_0(x)=-x+2 \frac{(x,\alpha)}{(\alpha,\alpha)}\alpha$ is a linear reflection with $\alpha$ a nontrivial weight of
$\pi_1$ or $\pi_2$ for type $\tilde {A_2}$ case, and of $\pi_\spin$ or $\pi_\st$ for $\tilde {C_2}$ case, and $\beta$ is chosen from the nontrivial weights of the other representation such that $(\alpha, \beta)$ is maximal. Thus $\alpha$ and $\beta$ form a set of fundamental weights. 
Note that all glide reflections in $\Gamma$ have the same linear part, namely the reflection  $\sigma_0$. Thus $\alpha$ is %uniquely determined for $\tilde A_2$ case, and unique up to sign for $\tilde{C}_2$ case.  
unique up to sign. %\kang{\xout{and $\beta$ is unique once $\alpha$ is chosen.}} 
Note also that $a,b$ are integers such that $a \alpha + b \beta \in \Lambda_r$ and $\sigma^2(x)=x+ k \alpha$, where
\begin{eqnarray}\label{k}
k = 2a + 2 \frac{(\alpha,\beta)}{(\alpha,\alpha)}b = 2 a+ \n b
\end{eqnarray}
with
\begin{eqnarray}\label{gamman}
 \n = \frac{2(\alpha,\beta)}{(\alpha,\alpha)}
 \end{eqnarray}
depending only on the group $\Gamma$. Replacing %$t$ by $t^{-1}$ and $\sigma$ by $\sigma^{-1}$
$\alpha$ by $-\alpha$ and $\beta$ by $-\beta$ if necessary, we may assume $k>0$.
The positive integer $k$ is independent of the choice of $\sigma$ and $t$ and depends only on $\Ga$. Hence we denote it by $\k$. %\kang{Moreover, when $\alpha$ is a fundamental weight of $\pi$, call $\Ga$ together $\{\sigma,t\})$ of type $\pi$. By abuse of notation, we also call $\Ga$ is of type $\pi$.}
With the choice of $t$ and $\sigma$ fixed, the positivity of $k$ determines $\alpha$, which is a weight of a unique $\pi$. We say that $\Ga$ is of type $\pi$.
In this case, $\n$ is independent of  $\alpha \in \wt'(\pi)$ and it will also be denoted by $n_\pi$.

A fundamental domain for $\langle \sigma^2 \rangle \bs \A$ is the strip
$$\A_{\sigma^2} = \{x\alpha + y \beta :~ x\in [0,\k), y\in \R\},$$
 and a fundamental domain  for $\langle \sigma \rangle \bs \A$ within $\A_{\sigma^2}$ is given by
$$\A_\sigma:= \{x\alpha + y \beta : ~ x\in [0,\k), y<\frac{b}{2}\} \bigcup \{ x\alpha + \frac{b}{2}\beta:~ x \in [0, \frac{\k}{2})\}.$$

For our choice of $G$, the possible scenarios are as follows:

\begin{itemize}
\item [(a)] $G$ is of type $\tilde{A}_2$ with $\alpha \in \wt(\pi_1) \cup \wt(\pi_2)$.
Then $\n=n_{\pi_1}=n_{\pi_2} =1$ and $\Lambda_r =\{ s_1 \alpha + s_2 \beta ~:~ s_1  \equiv s_2 \mod 3\}.$

\item [(b)] $G$ is of type $\tilde{C}_2$  with $\alpha \in \wt'(\pi_\st)$ (and hence $\beta \in \wt(\pi_\spin)$).
Then $\n = n_{\pi_\st} = 1$ and $\Lambda_r =\Z \alpha \oplus 2\Z \beta.$

\item [(c)]  $G$ is of type $\tilde{C}_2$  with $\alpha \in \wt(\pi_\spin)$ (and hence $\beta \in \wt'(\pi_\st)$).
Then $\n = n_{\pi_\spin}=2$ and $\Lambda_r =2\Z \alpha \oplus \Z \beta.$
\end{itemize}

Let $L_0$ be the glide reflection axis of $\sigma$ consisting of all points $x$ satisfying $\sigma(x)=x+ \frac{\k}{2} \alpha$. Then $L_0 = \frac{b}{2}\beta + \R \alpha$.
Note that $b$ is odd if and only if $L_0$ is irrational. In this case, we also call $\sigma$ irrational.

\begin{proposition} \label{prop3}
The following statements are equivalent:
\begin{enumerate}
\item[(1)] $\sigma$ is irrational;
\item[(2)] $t \sigma$ is irrational;
\item[(3)] $\frac{\k}{\n}$ is an odd integer.
\end{enumerate}
Furthermore, when $G$ is of type $\tilde{C}_2$ and $\n=1$, then $\sigma$ and $t \sigma$ are always rational.
\end{proposition}
\begin{proof}
 Suppose $\n=1$. Then $\k=2a+ \n b \equiv \n b \mod 2 \n $. Thus $\sigma$ is irrational if and only if $\frac{\k}{\n}$ is odd. In case that $G$ is of type $\tilde{C}_2$, $b$ is always even and hence $\sigma$ is rational.

Suppose $\n=2$. Since $\sigma \in \Lambda_r \rtimes W$, $a$ is even and $\k=2a+ \n b \equiv \n b \mod 2 \n.$
Thus $\sigma$ is irrational if and only if $\frac{\k}{\n}$ is odd, that is, (1) and (3) are equivalent.

Since $\sigma^2=(t \sigma)^2$, the above argument also holds for $t \sigma$ by (3).
\end{proof}

For $v = c \alpha + d \beta \in \Lambda_r$ with $d\neq 0$ and $\ga \in \Ga$, let
$$ \Lambda(\ga,v) = \{ x \in \Lambda :~\ga(x)=x+v\}.$$
We proceed to count the cardinality of $\Lambda(\ga,v)$, which will be used in the comparison of Langlands $L$-functions and the zeta function
of geodesic walks in the next subsection.

\begin{proposition}\label{prop7}
For an odd integer $m$ and $v= c \alpha +d \beta \in \Lambda_r$, the set $\A_\sigma \cap \Lambda(\sigma^m,v)$ is non-empty if and only if
$d>0$ and $(v,\alpha)= \frac{\k m}{2}(\alpha,\alpha)$, in which case it has cardinality $\k$. Moreover  $ \A_\sigma \cap \Lambda(\sigma^m,v)$ and
 $ \A_{t\sigma} \cap \Lambda( (t\sigma)^m,v)$ have the same cardinality.
\end{proposition}
\begin{proof}
Suppose there is an element $x \in \A_\sigma \cap \Lambda(\sigma^m,v)$.
Recall that $L_0 = \frac{b}{2}\beta + \R \alpha$ and $\sigma$ (and $\sigma^m$ for odd $m$) maps $s \beta + \R \alpha$ to $(b-s)\beta + \R \alpha$.
Comparing the coefficients of $\beta$ in $\sigma^m(x) = x+v$, we get $d>0$ because $x \in \A_{\sigma}$. On the other hand, from
\begin{align*} x+ m\k \alpha &= \sigma^{2m}(x) = \sigma^m(x+v) = \sigma^m(x)+\sigma_0(v)\\
&=x+v+\sigma_0(v)= x+ \frac{2(v,\alpha)}{(\alpha,\alpha)}\alpha
\end{align*}
we conclude that $(v,\alpha)= \frac{\k m}{2}(\alpha,\alpha)$.

Conversely, suppose $d>0$ and $(v,\alpha)= \frac{\k m}{2}(\alpha,\alpha)$.
Observe that
$$ 2c + d \n = \frac{2(v,\alpha)}{(\alpha,\alpha)}= \k m = m(2a+b \n ).$$
Observe that $b\equiv d \mod 2$ if $\n=1$. If $\n=2$, then $c$ and $a$ are both even. Therefore
$$ d \n \equiv 2c + d \n = m(2a+b \n ) \equiv m b \n \equiv b \n  \mod 2 \n.$$
We conclude that $d \equiv b \mod 2$  in both cases and $L=\frac{(b-d)}{2} \beta+\R \alpha$ is rational.
To finish the proof of the first assertion, we claim that $\Lambda(\sigma^m,v)=L \cap \Lambda$. If so, then
$$\A_\sigma \cap  \Lambda(\sigma^m,v) = \A_\sigma \cap L \cap \Lambda = \{ n \alpha + \frac{b-d}{2}\beta ~:~ n=0,  \dots, \k-1\}$$
has cardinality $\k$. To prove the claim, note first that for $x \in  \Lambda(\sigma^m,v)$, $\sigma^m(x)=x+v = x + c \alpha +d \beta$. Thus $x \in \frac{b-d}{2}\beta+\R \alpha = L$. Conversely,
for $x \in L$, $\sigma^m(x)-x-v \in \R \alpha$.
{Note that for all $ y \in \A$,
$$(\sigma(y)-y,\alpha)=\frac{\k}{2}(\alpha,\alpha)\quad \mbox{and}\quad (\sigma^m(y)-y,\alpha)=\frac{m\k}{2}(\alpha,\alpha) = (v, \alpha).$$
Thus, for $y=x$, $\sigma^m(x)-x = v$. Therefore the claim holds. }

Since $(t\sigma)^2(x)=\sigma^2(x)=x+ \k \alpha$, the same result also holds if we replace $\sigma$ by $t\sigma$.
\end{proof}

%&=\sum_{\la \in \wt'(\pi)}  \sum_{\ga \in [\Ga]} \# \{ x \in C_\Ga(\ga) \bs %\Lambda ~:~ \ga(x)  =  x+ n \lambda\}.

\subsection{ $N_n$ and $\tilde{N}_n$ revisited} For $\gamma \in \Gamma$, denote by $C_\Ga(\ga)$ the centralizer of $\ga$ in $\Ga$ and let $[\Ga]$ be a set of representatives of the conjugacy classes of $\Ga$. Then $\Gamma = \cup_{\gamma \in [\Gamma]} [\ga]$, where $[\ga] = \{g^{-1} \gamma g : g \in C_\Ga(\ga) \backslash \Ga\}$ is the conjugacy class of $\ga$. 

We begin by expressing $N_n$ in terms of the conjugacy classes of $\Ga$. 
Recall that
\begin{align*}
N_n
&=\sum_{\la \in \wt'(\pi)}  \sum_{\ga \in \Ga}  \# \{ x \in \Ga \bs \Lambda ~:~  \ga(x)  =  x+ n \lambda \} \\
&=\# \{ (\ga,x) \in \Ga \times (\Ga \bs \Lambda)~: ~\ga(x)- x \in n \cdot \wt'(\pi)\} \\
&=\sum_{\ga \in [\Ga]}\# \{(g^{-1} \ga g,x) \in [\ga] \times  (\Ga \bs \Lambda)~: ~g^{-1} \ga g(x) - x  \in n \cdot \wt'(\pi) \}.
\end{align*} 
Write $g_0$ for the linear part of %the affine transformation 
$g \in  \Ga$, which is an element of the Weyl group. Observe that  
$$ g^{-1} \ga g(x) = x + n \lambda \,\Leftrightarrow \, 
 \ga g(x) = g(x + n \lambda) = g(x) + n g_0(\lambda) = g(x)+n \lambda'.
$$
Here $\lambda' = g_0(\lambda)$ lies in $\wt'(\pi)$ if and only if $\lambda$ does.
Since $\Gamma$ acts freely on $\Lambda$,  %the set
for each $\ga \in [\Ga]$, the map $(g^{-1} \gamma g, x) \mapsto (\gamma, gx)$ from $[\ga] \times (\Ga \bs \Lambda)$ to $\{\ga\} \times (C_\Ga(\ga) \bs \Lambda)$ is a bijection. 
Therefore, we can re-express $N_n$ as 

\begin{align*}
N_n &=\sum_{\ga \in [\Ga]}\# \{ (\ga,x)~:~  x \in C_\Ga(\ga) \bs \Lambda , \ga(x)- x \in n \cdot \wt'(\pi)\}\\
&=\sum_{\la \in \wt'(\pi)}  \sum_{\ga \in [\Ga]} \# \{ x \in C_\Ga(\ga) \bs \Lambda ~:~ \ga(x)  =  x+ n \lambda\}.
\end{align*}

Next, we proceed to express $\tilde{N}_n$ in a similar way. Recall that  $\tilde{N}_n$ is the number of
 closed geodesic walks in $C_\pi$ of normalized length $n$, and a closed walk in $C_\pi$ is a closed geodesic walk if and only if its twice repetition also lies in $C_\pi$. Therefore
\begin{align*}
\tilde{N}_n
&=\sum_{\la \in \wt'(\pi)}  \sum_{\ga \in [\Ga]} \# \{ x \in C_\Ga(\ga) \bs \Lambda ~:~ \ga(x)  =  x+ n \lambda , \ga^2(x)=x+ 2 n\lambda\}.
\end{align*}

Note that for $\ga$ a translation or $\ga$ a glide reflection with reflection axis parallel to $\lambda$, the condition $\ga(x)=x+n\la$ always implies $\ga^2(x)=x+2 n \la$. On the other hand, for $\ga$ a glide reflection with reflection axis not parallel to $\lambda$, the relation $\ga(x)=x+n\la$ implies $\ga^2(x) \neq x+2 n \la$.

Combined with the fact that $C_\Ga(\sigma^m)= \langle \sigma \rangle$ and  $C_\Ga( (t\sigma)^m)= \langle t\sigma \rangle$ for all odd integers $m$, we conclude
\begin{align*}
N_n-\tilde{N_n}&=\sum_{\la \in \wt'(\pi), \la \neq \pm \alpha}~\sum_{m ~\mathrm{odd}} ~\sum_{\ga = \sigma, ~t \sigma} \# \{ x \in  \langle \ga \rangle \bs \Lambda ~:~ \ga^m(x)  =  x+ n \lambda\}\\
&= \sum_{\la \in \wt'(\pi),\la \neq \pm \alpha}~ \sum_{m~\mathrm{odd}}  ~\sum_{\ga = \sigma, ~t \sigma}  |\A_\gamma \cap \Lambda(\gamma^m, n \lambda)| \\
&= \sum_{\la \in \wt'(\pi),\la \neq \pm \alpha} ~\sum_{m~\mathrm{odd}} 2 |\A_\sigma \cap \Lambda(\sigma^m, n \lambda)|
\end{align*}
 by the second assertion of Proposition \ref{prop7}.

It follows from the first assertion of Proposition \ref{prop7} that, for $v= n \lambda$ with $|\A_\sigma \cap \Lambda(\sigma^m, n \lambda)| \ne 0$, we have $n (\la, \alpha) = (v,\alpha)=\frac{m \k}{2}(\alpha,\alpha)$ which implies $(\lambda, \alpha) \neq 0$.
On the other hand, for $\lambda \neq \pm \alpha$ and $(\lambda, \alpha) \neq 0$, a case-by-case analysis shows that $\frac{(\alpha,\alpha)}{2(\lambda,\alpha)} = \pm 1/\n$.
Therefore, when $|\A_\sigma \cap \Lambda(\sigma^m, n \lambda)| \ne 0$, we have $n= m \k/\n$  and  $m$ is positive and odd.

Recall from Proposition \ref{prop7} that $ |A_\sigma \cap \Lambda(\sigma^m, n \lambda)|= \k$ if it is nonzero, and this happens when $\lambda = c \alpha +d \beta$ satisfies the additional condition $d>0$.
Let $\wt_\Ga^+(\pi)$ consist of $\la \in \wt'(\pi)$ satisfying
$(\la,\alpha)\neq 0$  and $\la = c \alpha + d \beta $ with $d>0$. Therefore 
$$ |A_\sigma \cap \Lambda(\sigma^m, n \lambda)|= \begin{cases} \k  &\mbox{  if } \lambda \in \wt_\Ga^+(\pi) \mbox{~and~} n= m\k/\n, \\
0 & \mbox{ otherwise.}
\end{cases}
$$ 
Then we have
\begin{eqnarray*}
\sum_{n=1}^{\infty}(N_n-\tilde{N_n})\frac{u^n}{n}&=&\sum_{\la \in \wt_\Ga^+(\pi)} ~\sum_{m>0~\mathrm{odd}} \frac{2\k}{(m\k/\n)}   u^{ m\k/\n}\\
&=&\sum_{\la \in \wt_\Ga^+(\pi)} ~\sum_{r=0}^\infty \frac{2\n}{(2r+1)}   u^{(2r+1)\k/\n}\\
&=& |\wt_\Ga^+(\pi)| \n \log\frac{1+u^{\k/\n }}{1-u^{\k/\n}}.
\end{eqnarray*}

When $G$ is of type $\tilde{A}_2$, $\pi = \pi_1$ or $\pi_2$,
$n_\Ga=1$, and
$|\wt_\Ga^+(\pi)|= 1$.
When $G$ is of type $\tilde{C}_2$, $\pi = \pi_\spin$ or $\pi_\st$.
Then $|\wt_\Ga^+(\pi)|=0$ if $\Gamma$ is of type $\pi$, and $|\wt_\Ga^+(\pi)|=2$ otherwise.

For convenience, let
$$ \delta(\pi, \Ga) =
\begin{cases}
1 & \mbox{if $G$ is of type $\tilde{A_2}$,} \\
2 &\mbox{if $\Ga$ is not of type $\pi$ and $G$ is of type $\tilde{C_2}$,}\\
0 & \mbox{if $\Ga$ is of type $\pi$ and $G$ is of type $\tilde{C_2}$}.
\end{cases}
$$

The above discussion proves the following relation between the Langlands $L$-functions and the zeta function of geodesic walks of the Klein bottle $\A_\Ga$.
\begin{theorem} \label{maintheorem1}
Suppose $\A_\Ga$ is a Klein bottle, $\pi \in \{ \pi_1, \pi_2 \}$ if $G$ is of type $\tilde{A}_2$, and $\pi \in \{ \pi_\spin, \pi_\st \}$ if $G$ is of type $\tilde{C}_2$. Then
$$(1-u)^{\epsilon(\pi) N} L(\Ga \bs W_\ext,\pi,u^{})= Z(\A_\Ga,\pi,u)\left(\frac{1+u^{\k/n_{\Ga}}}{1-u^{\k/n_{\Ga}}}\right)^{n_{\Ga}
\delta(\pi, \Ga) }.$$ Here $\epsilon(\pi)$ is the multiplicity of the trivial weight of $\pi$ and $N$ is the cardinality of $\Ga \bs W_{\ext}/W$.
\end{theorem}

\subsection{Comparing zeta functions of geodesic walks of Klein bottles and those of their two-fold coverings}
As discussed in \S 4.1, the translations  in $\Ga$ form a subgroup $\Ga_0$ of index $2$. Let $L_0$ and $L_1$ be the glide reflection axis of $\sigma$ and $t\sigma$, respectively. Observe that if $t$ is translation by a vector $v \in \Lambda_r$, then $L_1$ is $L_0$ shifted by $\frac{1}{2}v$. These are the only two distinct glide reflection axes modulo the actions of $\Ga_0$ on $\A$. The finite quotient $\A_{\Ga_0}$ is a two-fold unramified cover of $\A_\Ga$. Given a primitive closed geodesic $\P$ (not necessary rational) in $\A_\Ga$, then there is either one primitive closed geodesic with twice the length of $\P$ or two disjoint primitive closed geodesics of the same length in $\A_{\Ga_0}$ which project to $\P$. We say that $\P$ is inert or split in $\A_{\Ga_0}$ accordingly.
Observe that $\P$ is inert in $\A_{\Ga_0}$ if and only if
any lifting of $\P$ in $\A$ is contained in $L_0$ or $L_1$ modulo $\Ga_0$ and the end point differs from the starting point by a glide reflection in $\Ga$. Furthermore all inert primitive closed geodesics in $\A_{\Ga}$  are  rational
if and only if $L_0$ and $L_1$ are rational.

Express the zeta function of $\pi$-geodesic walks on $\A_{\Ga}$ as
$$ Z(\A_\Ga,\pi, u ) = \prod_{\P~ \mathrm{split}} (1-u^{l(\P)})^{-1}~\prod_{\P~ \mathrm{inert}} (1-u^{l(\P)})^{-1}, $$
where $\P$ runs over rational primitive closed $\pi$-geodesics in $\A_\Ga$.
Note that the second product is nontrivial only when the direction of $L_0$ is parallel to a weight in $\pi$, in other words, either $\alpha$ or $-\alpha \in \wt'(\pi)$. For an inert $\P$, its length $l(\P)$ is equal to $\k/2$, half of the length of the primitive closed geodesic in $\A_{\Ga_0}$ projected to $\P$.  In particular, $\k$ is always even if an inert rational $\P$ exists.

We  may express the zeta function of $\pi$-geodesic walks on $\A_{\Ga_0}$ as a product over the rational primitive closed $\pi$-geodesics $\P$ in $\A_{\Ga}$:
\begin{eqnarray*}
Z(\A_{\Ga_0}, \pi, u )&=& \left(\prod_{\P~\mathrm{split}} (1-u^{l(\P)})^{-1} \right)^{2} \prod_{\P~ \mathrm{inert}} (1-u^{2l(\P)})^{-1}\\
&=& Z(\A_\Ga,\pi, u)^{2} \prod_{\P~\mathrm{inert}} \left(\frac{1-u^{\k/2}}{1+u^{\k/2}}\right).
\end{eqnarray*}
Note that the cardinality of $\{\pm \alpha\} \cap \wt'(\pi)$ is just $2-\delta(\pi,\Ga)$. Let $m_\Ga$ denote the number of $L_1$ and $L_0$ which are rational. Thus $m_\Ga = 2$ or $0$ because, by Proposition 4.4, $L_1$ and $L_0$ have the same rationality. Then we can rewrite the above result as
\begin{equation}\label{eq1}
Z(\A_{\Ga}, \pi, u )= Z(\A_{\Ga_0},\pi, u)^{1/2}  \left(\frac{1+u^{\k/2}}{1-u^{\k/2}}\right)^{m_\Ga(2- \delta(\pi,\Ga))/2}.
\end{equation}
Here the square root is chosen so that $Z(\A_\Ga,\pi, 0) = 1.$

Next, for a non-trivial weight $\lambda \in \wt(\pi)$, a $\lambda$-geodesic is called semi-rational if it is not rational but it becomes rational after translation by some $v \in \frac{1}{2}\Lambda$.
Note that $L_1$ and $L_0$ are either rational or semi-rational.
Consider the zeta function of $\lambda$-geodesic walks defined by
$$ Z(\A_{\Ga}, \lambda, u )  = \prod_{\P}(1- u^{l(\P)})^{-1},$$
where $\P$ runs through all primitive closed rational $\lambda$-geodesics in $\A_\Ga$.
Similarly, we define the zeta function
$Z_{\semi}(\A_{\Ga}, \lambda, u )$ by taking product over semi-rational $\lambda$-geodesics in the above definition.

It is clear that on the torus $A_{\Ga_0}$, translation by a $v \in \frac{1}{2}\Lambda \smallsetminus \Lambda$ induces a bijection between rational $\lambda$-geodesics and semi-rational $\lambda$-geodesics.
Thus,
$$ Z(\A_{\Ga_0}, \lambda, u ) =  Z_{\semi}(\A_{\Ga_0}, \lambda, u ).$$

For the Klein bottle $A_\Ga$, when $\lambda \neq \pm \alpha$, there are no inert geodesics in the direction of  $\lambda$. In this case we have
$$ Z(\A_{\Ga}, \lambda, u ) =  Z(\A_{\Ga_0}, \lambda, u )^{1/2} =  Z_{\semi}(\A_{\Ga_0}, \lambda, u )^{1/2}= Z_{\semi}(\A_{\Ga}, \lambda, u ).$$
When $\lambda = \pm \alpha$, the same argument as (\ref{eq1}) implies that
$$
Z(\A_{\Ga}, \lambda, u )= Z(\A_{\Ga_0},\lambda, u)^{1/2}  \left(\frac{1+u^{\k/2}}{1-u^{\k/2}}\right)^{m_\Ga/2}.
$$
Note that the number of $L_1$ and $L_0$ being semi-rational is equal to $2-m_\Ga$. Thus
\begin{eqnarray*}
Z_{\semi}(\A_{\Ga}, \lambda, u )&=& Z_{\semi}(\A_{\Ga_0},\lambda, u)^{1/2}  \left(\frac{1+u^{\k/2}}{1-u^{\k/2}}\right)^{(2- m_\Ga)/2}.
\end{eqnarray*}

We summarize the above discussion in the following theorem.
\begin{theorem} \label{maintheorem2}
Under the above notation, we have
$$Z_{\semi}(\A_{\Ga_0},\pi,u)= Z(\A_{\Ga_0}, \pi, u) $$
and
$$Z_{\semi}(\A_\Ga, \pi,  u) = Z(\A_\Ga,\pi, u) \left(\frac{1+u^{\k/2}}{1-u^{\k/2}}\right)^{(2-\delta(\pi,\Ga)) (1-m_\Ga)}.$$
\end{theorem}
Note that $Z(\A_\Gamma,\pi,u)$ is always a rational function in $u$ by definition and the above theorem implies that $Z_{\semi}(\A_\Gamma,\pi,u)$ is a rational function in $u^{1/2}$.

\section{Zeta functions of geodesic galleries}

A semi-rational geodesic is the middle line of the strip bounded by two closest rational geodesics parallel to it.
We shall also consider a zeta function counting these strips.

\subsection{$\pi$-geodesic galleries}
Fix a minuscule or quasi-minuscule representation $\pi$ of $\hat{G}(\C)$ as before.
Let $\lambda$ be a nontrivial weight of $\pi$ and $v$ be an element of $\Lambda$. A $\pi$-chamber $C_{\lambda,v}$ is the union of chambers in $\A$ containing {$(v,\lambda+v)$}. The directed edge from $v$ to $\lambda+v$ is called the central edge of $C_{\lambda,v}$. As we shall see, when $G=$PGSp$_4$, the same union of chambers forms more than one $\pi$-chamber, with different central edges.

A sequence $(\cdots, C_{-1},C_0,\cdots)$ of $\pi$-chambers is called a $\pi$-geodesic gallery if the following two conditions hold:
\begin{enumerate}
\item For all $i<j$, the convex hull of $C_i$ and $C_j$ is the union $C_i \cup C_{i+1} \cup \cdots \cup C_j$;
\item For $i$ not equal to the highest index, the terminal vertex of the central directed edge of $C_i$ is equal to the initial vertex of the central directed edge of $C_{i+1}$.
\end{enumerate}

We explain these $\pi$-geodesic galleries in more detail.

(I) For $G=$PGL$_3$, let $\pi = \pi_1$ with $\wt(\pi_1)$ and a $\pi_1$-chamber as shown below:
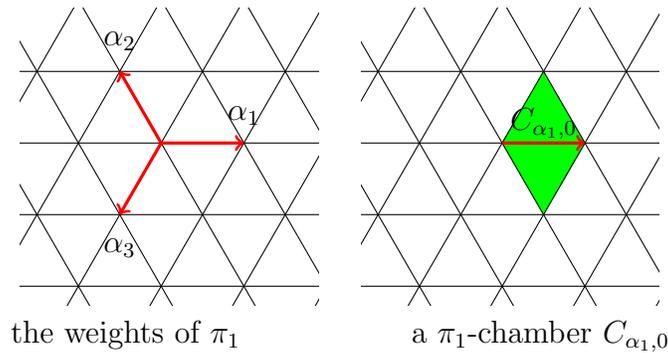
\begin{center}
\begin{tikzpicture}[scale =1.1]
\clip (-3.7,-3.7) rectangle ++ (3.6,3.6);
\foreach \y in {-10,-9,...,10}
{
\draw plot (\x+\y,1.732*\x);
\draw plot (-\x+\y,1.732*\x);
\draw plot (-\x,0.866*\y);
}
\draw[red, very thick,->] (-2,-1.732)--(-1,-1.732);
\draw[red, very thick,->] (-2,-1.732)--(-2.5,-2.598);
\draw[red, very thick,->] (-2,-1.732)--(-2.5,-.866);
\node at (-1,-1.4) {$\alpha_1$};
\node at (-2.5,-.5) {$\alpha_2$};
\node at (-2.5,-3) {$\alpha_3$};
\end{tikzpicture}
\hspace{.3cm}
\begin{tikzpicture}[scale =1.1]
\clip (-3.7,-3.7) rectangle ++ (3.6,3.6);
\filldraw [green] (-2,-1.732)--(-1.5,-2.598)--(-1,-1.732)--(-1.5,-.866)--(-2,-1.732);
\foreach \y in {-10,-9,...,10}
{
\draw plot (\x+\y,1.732*\x);
\draw plot (-\x+\y,1.732*\x);
\draw plot (-\x,0.866*\y);
}
\draw[red, very thick,->] (-2,-1.732)--(-1,-1.732);
 \node at (-1.5,-1.5) {$C_{\alpha_1,0}$};
\end{tikzpicture}\\
the weights of $\pi_1$ \hspace{2cm} a $\pi_1$-chamber $C_{\alpha_1,0}$
\end{center}
Then there are two $\pi_1$-geodesic galleries in $\A$ containing $C_{\alpha_1,0}$:

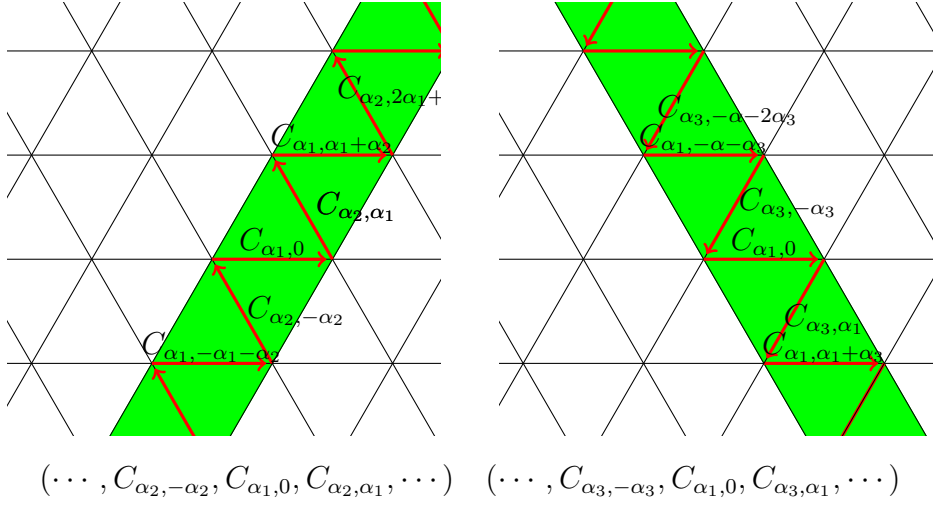
\begin{center}
\begin{tikzpicture}[scale =1.6]
\clip (-3.7,-3.2) rectangle ++ (3.6,3.6);
\filldraw [green] (-11,-17.32)--(-10,-17.32)--(10,17.32)--(9,17.32)--(-11,-17.32);
\foreach \y in {-10,-9,...,10}
{
\draw plot (\x+\y,1.732*\x);
\draw plot (-\x+\y,1.732*\x);
\draw plot (-\x,0.866*\y);
\draw[red, very thick,->] (-2+.5*\y,-1.732+.866*\y)--(-1.05+.5*\y,-1.732+.866*\y);
\draw[red, very thick,->] (-1+.5*\y,-1.732+.866*\y)--(-1.475+.5*\y,-.91+.866*\y);
}

 \node at (-1.5,-1.6) {$C_{\alpha_1,0}$};
 \node at (-.8,-1.3) {$C_{\alpha_2,\alpha_1}$};
 \node at (-2,-2.466) {$C_{\alpha_1,-\alpha_1-\alpha_2}$};
  \node at (-1.3,-2.166) {$C_{\alpha_2,-\alpha_2}$};
  \node at (-.8,-1.3) {$C_{\alpha_2,\alpha_1}$};
  \node at (-1,-.734) {$C_{\alpha_1,\alpha_1+\alpha_2}$};
    \node at (-.4,-.34) {$C_{\alpha_2,2\alpha_1+\alpha_2}$};
\end{tikzpicture}
\hspace{.5cm}
\begin{tikzpicture}[scale =1.6]
\clip (-3.7,-3.2) rectangle ++ (3.6,3.6);
\filldraw [green] (7,-17.32)--(8,-17.32)--(-12,17.32)--(-13,17.32)--(7,-17.32);
\foreach \y in {-10,-9,...,10}
{
\draw plot (\x+\y,1.732*\x);
\draw plot (-\x+\y,1.732*\x);
\draw plot (-\x,0.866*\y);
\draw[red, very thick,->] (-2-.5*\y,-1.732+.866*\y)--(-1.05-.5*\y,-1.732+.866*\y);
\draw[red, very thick,->] (-1-.5*\y,-1.732+.866*\y)-- (-1.475-.5*\y,-2.555+.866*\y);
}

 \node at (-1.5,-1.6) {$C_{\alpha_1,0}$};
 \node at (-1,-2.466) {$C_{\alpha_1,\alpha_1+\alpha_3}$};
 \node at (-2,-.734) {$C_{\alpha_1,-\alpha-\alpha_3}$};
 \node at (-1.3,-1.267) {$C_{\alpha_3,-\alpha_3}$};
 \node at (-1,-2.233) {$C_{\alpha_3,\alpha_1}$};
 \node at (-1.8,-.5) {$C_{\alpha_3,-\alpha-2\alpha_3}$};

\end{tikzpicture}
$$(\cdots,C_{\alpha_2,-\alpha_2}, C_{\alpha_1,0},C_{\alpha_2,\alpha_1},\cdots)\quad (\cdots,C_{\alpha_3,-\alpha_3}, C_{\alpha_1,0},C_{\alpha_3,\alpha_1},\cdots)$$

\end{center}
Recall that $\pi_2$ has weights $-\alpha_1, -\alpha_2$, and $-\alpha_3$.
Notice that the middle line of a $\pi_1$-geodesic gallery parallel to its sides is a semi-rational $\pi_2$-geodesic, and similarly the middle line of a $\pi_2$-geodesic gallery is a semi-rational $\pi_1$-geodesic.

(II) For $G=$PGSp$_4$, the weights of $\pi_\spin$ and $\pi_\st$ are as follows.
\begin{center}
\begin{tikzpicture}[scale =1]
\clip (-3.7,-3.7) rectangle ++ (3.6,3.6);
\foreach \y in {-10,-9,...,10}
{
\draw plot (\x,\y);
\draw plot (\y,\x);
\draw plot (\x+\y,\x);
\draw plot (-\x+\y,\x);
}
\draw[red, very thick,->] (-2,-2)--(-1,-2);
\draw[red, very thick,->] (-2,-2)--(-3,-2);
\draw[red, very thick,->] (-2,-2)--(-2,-1);
\draw[red, very thick,->] (-2,-2)--(-2,-3);
 \node at (-1.5,-1.8) {$\alpha$};

\end{tikzpicture}
\hspace{.3cm}
\begin{tikzpicture}[scale =1]
\clip (-3.7,-3.7) rectangle ++ (3.6,3.6);
\foreach \y in {-10,-9,...,10}
{
\draw plot (\x,\y);
\draw plot (\y,\x);
\draw plot (\x+\y,\x);
\draw plot (-\x+\y,\x);
}
\draw[blue, very thick,->] (-2,-2)--(-1,-1);
\draw[blue, very thick,->] (-2,-2)--(-3,-3);
\draw[blue, very thick,->] (-2,-2)--(-1,-3);
\draw[blue, very thick,->] (-2,-2)--(-3,-1);
\filldraw [blue] (-2,-2) circle (4pt);
\node at (-1.5,-1.1) {$\beta$};
\node at (-1.5,-2.9) {$\beta'$};
\end{tikzpicture}\\
the weights of $\pi_\spin$ \hspace{1cm} the weights of $\pi_\st$
\end{center}

An example of a $\pi_\spin$-chamber is $C_{\alpha,0}$ shown below. There are two $\pi_\spin$-geodesic galleries containing $C_{\alpha, 0}$, determined by the $\pi_\spin$-chamber following it. One is as shown below with the central edges of the chambers adjacent to $C_{\alpha,0}$ pointing upward, the other one has the immediate adjacent central edges pointing downward. Consequently the middle lines of the two $\pi_\spin$-geodesic galleries containing $C_{\alpha, 0}$ are perpendicular to each other.
\begin{center}
\begin{tikzpicture}[scale =1]
\clip (-3.7,-3.7) rectangle ++ (3.6,3.6);
\filldraw [green] (-2,-2)--(-1.5,-1.5)--(-1,-2)--(-1.5,-2.5)--(-2,-2);
\foreach \y in {-10,-9,...,10}
{
\draw plot (\x,\y);
\draw plot (\y,\x);
\draw plot (\x+\y,\x);
\draw plot (-\x+\y,\x);
}
\draw[red, very thick,->] (-2,-2)--(-1,-2);
 \node at (-1.5,-1.8) {$\alpha$};
\end{tikzpicture}
\hspace{.8cm}
\begin{tikzpicture}[scale =1]
\clip (-3.7,-3.7) rectangle ++ (3.6,3.6);
\filldraw [green] (-12,-12)--(10,10)--(11,10)--(-11,-12)--(-12,-2);
\foreach \y in {-10,-9,...,10}
{
\draw plot (\x,\y);
\draw plot (\y,\x);
\draw plot (\x+\y,\x);
\draw plot (-\x+\y,\x);
\draw[red, very thick,->] (-2+\y,-2+\y)--(-1.05+\y,-2+\y);
\draw[red, very thick,->] (-1+\y,-2+\y)--(-1+\y,-1.05+\y);
}
 \node at (-1.5,-1.8) {$\alpha$};
\end{tikzpicture}\\
$\qquad$ the $\pi_\spin$-chamber $C_{\alpha,0}$ \hspace{0.1cm} $\qquad $ a $\pi_\spin$-geodesic gallery $\qquad $
\end{center}
Note that both $C_{-\alpha, \alpha}$ and $C_{\alpha,0}$ are the union of the same chambers but their central directed edges are opposite.

Shown below are an example of a $\pi_\st$-chamber $C_{\beta, 0}$  and a $\pi_\st$-geodesic gallery containing it. Similar to  the previous case, there is another $\pi_\st$-geodesic gallery containing $C_{\beta, 0}$ so that the middle lines of the two galleries are perpendicular to each other. Different from the previous case, on the union of the chambers in $C_{\beta,0}$ there are four choices of the central directed edges giving rise to four distinct $\pi_\st$-chambers. Denote by $\beta'$ the weight of $\pi_\spin$ perpendicular to $\beta$ and pointing downward. Then one of the four $\pi_\st$-chambers $C'$ on the same square as $C_{\beta, 0}$ has the central directed edge a shift of $\beta'$. Observe that the two $\pi_\st$-geodesic galleries containing $C'$ have the middle lines coincide with those containing $C_{\beta, 0}$.

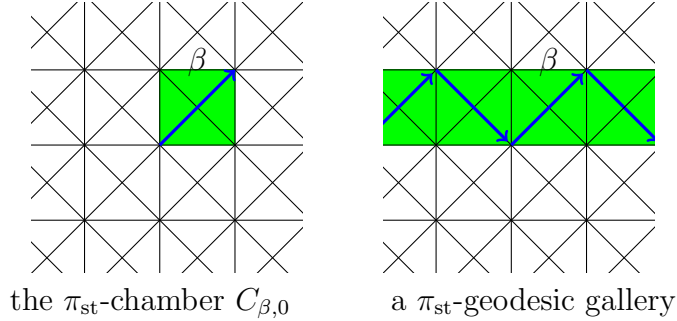
\begin{center}
\begin{tikzpicture}[scale =1]
\clip (-3.7,-3.7) rectangle ++ (3.6,3.6);
\filldraw [green] (-2,-2)--(-1,-2)--(-1,-1)--(-2,-1)--(-2,-2);
\foreach \y in {-10,-9,...,10}
{
\draw plot (\x,\y);
\draw plot (\y,\x);
\draw plot (\x+\y,\x);
\draw plot (-\x+\y,\x);
}
\draw[blue, very thick,->] (-2,-2)--(-1,-1);
\node at (-1.5,-.9) {$\beta$};
\end{tikzpicture}
\hspace{.8cm}
\begin{tikzpicture}[scale =1]
\clip (-3.7,-3.7) rectangle ++ (3.6,3.6);
\filldraw [green] (-10,-2)--(10,-2)--(10,-1)--(-10,-1)--(-10,-2);
\foreach \y in {-10,-9,...,10}
{
\draw plot (\x,\y);
\draw plot (\y,\x);
\draw plot (\x+\y,\x);
\draw plot (-\x+\y,\x);
\draw[blue, very thick,->] (-2+2*\y,-2)--(-1.05+2*\y,-1.05);
\draw[blue, very thick,->] (-1+2*\y,-1)--(2*\y-0.05,-1.95);
}
\node at (-1.5,-.9) {$\beta$};
\end{tikzpicture} \\
$\qquad$ the $\pi_\st$-chamber $C_{\beta,0}$ \hspace{0.1cm} $\qquad $ a $\pi_\st$-geodesic gallery $\qquad $
\end{center}

\noindent Note that the middle line of a $\pi_\st$-geodesic gallery is a semi-rational $\pi_\spin$-geodesic, and the middle line of a $\pi_\spin$-geodesic gallery is a semi-rational $\pi_\st$-geodesic.

We summarize the above discussion in
\begin{theorem}\label{middleline}
Let the set $\{\pi, \pi'\}$ be $\{\pi_1, \pi_2 \}$ for $G = {\rm PGL}_3$ and $\{\pi_\spin, \pi_\st \}$ for $G= {\rm PGSp}_4$. Then
the middle line of a $\pi$-geodesic gallery is a semi-rational $\pi'$-geodesic. Conversely, a semi-rational $\pi'$-geodesic is the middle line of a $\pi$-geodesic gallery. Furthermore, the map from the equivalence classes of $\pi$-geodesic galleries in $\A$ to the  semi-rational lines of type $\pi'$ in $\A$ sending the equivalence class of a $\pi$-geodesic gallery to its semi-rational middle line is a bijection except for $\pi = \pi_\st$, in which case it is a $2$-to-$1$ surjection.
\end{theorem}

\subsection{Zeta functions of geodesic galleries}
Define $\pi$-geodesic galleries as well as primitive and equivalent $\pi$-geodesic galleries on $\A_\Ga$ similar to what we did for walks.  A closed $\pi$-geodesic gallery $C$ in $\A_\Ga$ is called a geodesic gallery if the closed walks formed by the central directed edges of the $\pi$-chambers in $C$ as well as all geodesic galleries equivalent to it contain no backtracking.
The length $l(C)$ of a closed $\pi$-geodesic gallery $C$ in $\A_\Ga$ is defined to be the number of $\pi$-chambers contained in $C$. The middle line $c'$ of $C$ is a semi-rational closed $\pi'$-geodesic. Here $\{\pi, \pi'\}$ is as in
Theorem \ref{middleline}. It has the normalized length (with respect to the nonzero weight in $\pi'$) $l(c')$.
The ratio $l(C)/l(c')$ is a constant independent of the choice of $C$. Indeed, a case by case check shows that $l(C)/l(c')= 2/n_{\pi'}.$

The zeta function of $\pi$-geodesic galleries of $\A_\Ga$ is defined as
$$ Z_2(\A_\Ga, \pi ,u) = \prod_{C}(1- u^{l(C)})^{-1}, $$
where $C$ runs through the equivalence classes of all primitive closed $\pi$-geodesic galleries in $\A_\Ga$.

When $\A_\Ga$ is a Klein bottle, Theorem \ref{maintheorem1} gives the ratio of the Langlands L-function and the geometric zeta function as
\begin{eqnarray}\label{Loverzeta}
\frac{(1 - u)^{\epsilon(\pi)N}L(\Ga \bs W_\ext, \pi, u)}{Z(\A_\Ga, \pi, u)} =  \left(\frac{1+u^{\k/n_{\Ga}}}{1-u^{\k/n_{\Ga}}}\right)^{n_\Ga \delta(\pi,\Ga)}.
\end{eqnarray}
Setting $\delta(\pi, \Ga) = 0$ and interpreting the right hand side of (\ref{Loverzeta}) to be 1 when $\A_\Ga$ is a torus, we get that (\ref{Loverzeta}) also holds for the torus case by Theorem \ref{maintheorem0}. Next we compare {the zeta function of geodesic galleries with the zeta function of geodesic walks.}
Combining the above discussions with Theorems \ref{middleline} and \ref{maintheorem2}, we obtain the following result. Recall that $n_\pi = 1$ for $\pi = \pi_1, \pi_2, \pi_\st$, and $n_\pi = 2$ for $\pi = \pi_\spin$.
\begin{theorem}\label{zeta2andzeta}
Let $\{\pi, \pi'\}$ be $\{\pi_1, \pi_2 \}$ for $G = {\rm PGL}_3$ and $\{\pi_\spin, \pi_\st \}$ for $G= {\rm PGSp}_4$. Then
\begin{eqnarray*}
 Z_2(\A_\Ga, \pi ,u) &=& Z_{\semi}(\A_\Ga, \pi', u^{2/n_{\pi'}})^{n_{\pi'}}\\
&=& Z(\A_\Ga,\pi', u^{2/n_{\pi'}})^{n_{\pi'}} \left(\frac{1+u^{\k/n_{\pi'}}}{1-u^{\k/n_{\pi'}}}\right)^{\delta(\pi,\Ga) (1-m_\Ga)n_{\pi'}}.
\end{eqnarray*}
\end{theorem}
Here we have used the fact that $2 - \delta(\pi', \Ga) = \delta(\pi, \Ga)$ in case that $\A_\Ga$ is a Klein bottle.

The main goal of this subsection is to prove the following result relating the Langlands L-function, the zeta functions of geodesic walks and the zeta functions of geodesic galleries for $\A_\Ga$, similar to the known results for the quotients of the building attached to $G(F)$ obtained in \cites{KL1, KLW, FLW}.
\begin{theorem}\label{Aidentity}
Let $\Ga$ be a torsion-free subgroup of $W_\ext$ with $N=|\Ga \backslash W_\ext/W|.$ Let $\{\pi, \pi'\}$ be $\{\pi_1, \pi_2 \}$ for $G = {\rm PGL}_3$ and $\{\pi_\spin, \pi_\st \}$ for $G= {\rm PGSp}_4$. Denote by $\epsilon(\pi)$ the multiplicity of the trivial weight in $\wt(\pi)$. The following identity holds:
$$(1-u)^{\epsilon(\pi) N} L(\Ga \bs W_\ext,\pi , u) = \frac{ Z(\A_\Ga,\pi,u) Z(\A_\Ga,\pi', u^{2/n_{\pi'}})^{n_{\pi'}}}
{Z_2(\A_\Ga, \pi,-u)}.$$
More precisely, when $G$ is of type $\tilde{A}_2$,
we have
\begin{eqnarray}\label{A2}
L(\Ga \bs W_\ext,\pi_1 , u) = L(\Ga \bs W_\ext, \pi_2, u) = \frac{ Z(\A_\Ga,\pi_1,u) Z(\A_\Ga,\pi_2, u^2)}
{Z_2(\A_\Ga, \pi_1,  -u)}.
\end{eqnarray}
When $G$ is of type $\tilde{C}_2$, we have
\begin{eqnarray}\label{spin}
L(\Ga \bs W_\ext,\pi_{\spin} , u) = \frac{ Z(\A_\Ga,\pi_\spin,u)  Z(\A_\Ga,\pi_\st, u^2)}
{Z_2(\A_\Ga, \pi_\spin,  -u)}
\end{eqnarray}
and
\begin{eqnarray}\label{st}
(1-u)^{N} L(\Ga \bs W_\ext,\pi_\st, u) =\frac{  Z(\A_\Ga,\pi_\st,u) Z(\A_\Ga,\pi_\spin, u)^2}
{Z_2(\A_\Ga, \pi_\st,-u)}.
\end{eqnarray}
\end{theorem}

\begin{proof} In view of (\ref{Loverzeta}), it remains to show that the right hand side of (\ref{Loverzeta}) is equal to  the ratio $Z(\A_\Ga, \pi', u^{2/n_{\pi'}})^{n_\pi'}/Z_2(\A_\Ga, \pi, -u)$. When $\delta(\pi, \Ga) = 0$, $G$ is of type $\tilde{C_2}$ and $\Gamma$ is of type $\pi$. 
Since $\Gamma$ preserves the types of vertices and each side of a $\pi$-chamber connects vertices of different types, all closed $\pi$-geodesic galleries must have even lengths in order to be closed by $\Ga$. Therefore, $Z_2(\A_\Ga, \pi, u) = Z_2(\A_\Ga, \pi, -u)$. The desired equality follows from Theorem \ref{zeta2andzeta}.

Next assume $\delta(\pi, \Ga) \ne 0$. This happens when $G=$PGL$_3$, in which case $n_{\pi'} = n_\Ga = 1$, or $G=$PGSp$_4$ and $\Ga$ is not of type $\pi$, hence $\Ga$ is of type $\pi'$ so that $n_{\pi'} = n_\Ga$. Recall from Proposition \ref{prop3} that $\k/n_\Ga = \k/n_{\pi'}$ is odd if and only if the two glide reflection axes $L_0$ and $L_1$ are semi-rational, that is, $m_\Ga = 0$.
In this case Theorem \ref{zeta2andzeta} implies

$$ \frac{Z(\A_\Ga, \pi', (-u)^{2/n_{\pi'}})^{n_{\pi'}}}{Z_2(\A_\Ga, \pi, -u)} =  \left(\frac{1+u^{\k/n_{\pi'}}}{1-u^{\k/n_{\pi'}}}\right)^{\delta(\pi,\Ga)n_{\pi'}}.$$
Furthermore, $Z(\A_\Ga, \pi', (-u)^{2/n_{\pi'}})$ is an even function. This is obvious when $n_{\pi'} = 1$. When $n_{\pi'} = 2$, i.e. $\pi' = \pi_\spin$, a closed rational $\pi_\spin$-geodesic $c$ has even length because the two end points of each edge in $c$ are of different types, with one point being primitive special and the other  nonprimitive special. Therefore the above identity can be rewritten as
\begin{eqnarray}\label{quotient}
 \frac{Z(\A_\Ga, \pi', u^{2/n_{\pi'}})^{n_{\pi'}}}{Z_2(\A_\Ga, \pi, -u)} =  \left(\frac{1+u^{\k/n_{\Ga}}}{1-u^{\k/n_{\Ga}}}\right)^{\delta(\pi,\Ga)n_{\Ga}}.
\end{eqnarray}
\noindent This equality also holds for even $\k/n_{\pi'}$ by Theorem \ref{zeta2andzeta} because geodesic galleries in  $\A_\Ga$ have even lengths. This completes the proof of the theorem.
\end{proof}

\section{Zeta functions of finite quotients of buildings}

As before let $F$ be a non-Archimedean local field with $q$ elements in its residue field. Denote by $\O$ its ring of integers and $\varpi$ a fixed uniformizer. Zeta functions counting closed geodesics contained in the 1-skeleton of finite quotient complexes
arising from the Bruhat-Tits buildings of PGL$_3(F)$ and PGSp$_4(F)$ have been studied in
\cites{KLW, KL1, FLW}. Such a zeta function is expressed in two ways combinatorially: one in terms of edge adjacency operators, and the other in terms of vertex adjacency operators and the chamber adjacency operator. The establishment of these zeta identities there mostly relies on calculations with lattice models of the Bruhat-Tits buildings and knowledge from representation theory. In this section, we recast these expressions using group theoretic language as  we did for finite quotients of apartments in the previous sections. This will facilitate  comparison with results from section 5. We also establish a new zeta identity involving degree 5 standard L-function in the case of PGSp$_4(F)$, presented in Theorem \ref{deg5}.

In what follows, $G = $ PGL$_3$ or PGSp$_4$, and $K=G(\O)$ denotes the standard maximal compact subgroup of $G(F)$. The  associated building $\B$ is a contractible $2$-dimensional simplicial complex. Fix a discrete torsion-free cocompact subgroup $\Ga$ of $G(F)$ such that $\ord_F(\det \Ga) \subset 3\mathbb{Z}$ if $G = $PGL$_3$ and $\ord_F(\det \Ga) \subset 4\mathbb{Z}$ if $G = $PGSp$_4$. The action of $\Ga$ on $\B$ by left translation preserves the types of the vertices of the building. The quotient $\B_\Ga = \Ga \bs \B$ is a finite $2$-dimensional simplicial complex.

\subsection{$L$-functions and zeta functions of geodesic walks and galleries}
Let $\pi$ be a representation of the dual group $\hat G(\mathbb C)$. Denote by $\Irr^\unr(\Ga\backslash G(F))$ the set of irreducible unramified subrepresentations of $G(F)$ (with multiplicity) occurring in  $L^2(\Ga \bs G(F))$. A representation $\rho \in \Irr^\unr(\Ga \backslash G(F))$ is a constituent of the representation
parabolically induced from a character $\chi_\rho$ of the maximal split torus of $G(F)$.
The $L$-function of type $\pi$ associated to $\Ga \bs G(F)$ is defined as {a product of local
Langlands $L$-functions}
$$ L(\Ga \bs G(F), \pi, u) := \prod_{\rho \in \Irr^\unr(\Ga \backslash G(F))} L(\rho, \pi, u)= \prod_{\rho \in \Irr^\unr(\Ga \backslash G(F))}~ \prod_{\la \in \wt(\pi)} (1- \chi_\rho( \la) u )^{-1},$$
just like the apartment case (cf. \S 3.4).

The weights of $\pi$ on the standard apartment $\A$ of $\B$ are defined in \S2.3 and $\pi$-geodesics  on $\A$ and its quotients are defined in \S 3.1-3.2. Since a closed geodesic walk $c$ in $\B_\Ga$ lifts to a geodesic walk $c'$ on an apartment of $\B$ which is the image of $\A$ under left translation by some element $g \in G(F)$, we say that $c$ is a $\pi$-geodesic if $c'$ is a $g(\wt(\pi))$-geodesic in $g\A$.
Similarly, the zeta function of $\pi$-geodesic walks on $\B_\Ga$ is defined as
$$ Z(\B_\Ga,\pi,u) = \prod_{c}(1- u^{l(c)})^{-1}$$
where $c$ runs through all primitive rational closed $\pi$-geodesics in $\B_\Ga$, and the normalized length $l(c)$ is equal to the length of $c$ divided by the length $\ell(\pi)$ of the nontrivial weights of $\pi$.  Here the rationality condition means that the geodesic is contained in the $1$-skeleton of $\B_\Ga$, as in \S 3.3.

Similar to the apartment case, we also consider $\pi$-geodesic galleries in $\B_\Ga$. Like walks, a $\pi$-geodesic gallery in $\B$ is contained in an apartment $g(\A)$ for some $g \in G(F)$, and the various terminologies for $\pi$-geodesic galleries in the apartment case discussed in \S5 can be carried over to the building $\B$ and its quotients $\B_\Ga$. In particular, a closed $\pi$-geodesic gallery $C$ in $\B_\Ga$ means that its lifting $C'$ in $g(\A)$ is of type $g(\wt'(\pi))$.
Define the {zeta function of $\pi$-geodesic galleries} of $\B_\Ga$ to be
$$ Z_2(\B_\Ga, \pi ,u) = \prod_{C}(1- u^{l(C)})^{-1}, $$
where $C$ runs through the equivalence classes of all primitive closed $\pi$-geodesic galleries in $\B_\Ga$, and the length $l(C)$ is the number of $\pi$-chambers contained in a gallery in the class of $C$.

\subsection{Type $\tilde A_2$}
Suppose $G=\PGL_3$. We draw connections between the functions defined above and the functions occurring in the zeta identities in \cites{KL1, KLW}.
Let the representations $\pi_1$ and $\pi_2$ of SL$_3(\mathbb C) = \hat G(\mathbb C)$ be as in \S 5.1. These are minuscule representations. The weights in $\wt(\pi_1)$ can be identified with the vertices $\diag(\varpi, 1, 1)$, $\diag(1, \varpi, 1)$ and $\diag(1, 1, \varpi)$ in the standard apartment of $\B$, and the weights in $\wt(\pi_2)$ are opposite of those in $\wt(\pi_1)$.  This immediately implies $Z(\B_\Ga,\pi_1,u) = Z(\B_\Ga,\pi_2,u)$ and $Z_2(\B_\Ga,\pi_1,u) = Z_2(\B_\Ga,\pi_2,u) $.

Given a representation $\rho \in \Irr^\unr(\Ga \backslash G(F))$,
the character $\chi_\rho$ may be identified, up to permutation, with a triple of unramified characters $(\chi_1, \chi_2, \chi_3)$ of $F^\times$ with product equal to $1$. Then
$$\prod_{\la \in \wt(\pi_1)} (1 - \chi_\rho(\lambda)u) =  (1-\chi_1(\varpi)u)(1-\chi_2(\varpi)u)(1-\chi_3(\varpi)u).$$
As shown in \cites{KLW, KL1}, this gives rise to an expression for the Langlands $L$-function in terms of Hecke operators  $A_1$ and $A_2$  on $L^2(\Ga \bs G(F)/K)$ associated to the $K$-double cosets represented by $\diag(1, 1, \varpi)$ and $\diag(1, \varpi, \varpi)$, respectively:

$$L(\Ga \bs G(F), \pi_1, qu) = \frac{1}{\det(I-A_1u+q A_2u^2-q^3u^3I)}.$$

\noindent Further, it was shown in \cite{KL1} that
$$Z(\B_\Ga, \pi_1, u) =\frac{1}{ \det(I - L_Eu)} \quad {\rm{and}} \quad Z(\B_\Ga, \pi_2, u) = \frac{1}{\det(I - (L_E)^t u)},$$
where $E$ is a parahoric subgroup of $K$,  $L_E$ is the parahoric operator on $L^2(\Ga \bs G(F)/E)$ associated to the double coset $E\diag(1, 1, \varpi)E$, and $(L_E)^t$ is the transpose of $L_E$.   Likewise, the type 2 chamber zeta function $Z_{2,2}(\B_\Ga, u)$ in \cite{KL1}, which counts the number of closed
$\pi_1$-geodesic galleries in $\B_\Ga$, is nothing but $Z_2(\B_\Ga, \pi_1, u)$. As shown in \cite{KL1}, it can be expressed using a suitable Iwahori-Hecke operator $L_I$ acting on $L^2(\Ga \bs G(F)/I)$ associated to a double coset of the Iwahori subgroup $I$ of $K$. More precisely, we have

$$Z_{2,2}(\B_\Ga, u) = Z_2(\B_\Ga, \pi_1, u) = \frac{1}{\det(I - L_Iu)}.$$

The zeta function of geodesic walks of $\B_\Ga$, denoted $Z(\B_\Ga, u)$, counts the number of rational closed geodesic walks in
$\B_\Ga$ in algebraic length; these closed walks turn out to be either $\pi_1$-walks or $\pi_2$-walks. Hence we have
$$Z(\B_\Ga, u) = Z(\B_\Ga, \pi_1, u)Z(\B_\Ga, \pi_2, u^2) = \frac{1}{\det(I-L_E u)\det(I-(L_E)^t u^2)}.$$
It was proved in \cites{KL1, KLW} that $Z(\B_\Ga, u)$ affords another expression in terms of the Hecke operators $A_1, A_2$ and the Iwahori-Hecke operator $L_I$:

\begin{equation*}
Z(\B_\Ga, u) %= \frac{1}{\det(I-L_E u)\det(I-(L_E)^t u^2)} \\
= \frac{(1-u^3)^{\chi(\B_\Ga)}}{\det(I-A_1u+q A_2u^2-q^3u^3I) \det(I + L_Iu)},
\end{equation*}
where $\chi(\B_\Ga)$ is the Euler characteristic of $\B_\Ga$. Combining the above two expressions for $Z(\B_\Ga, u)$, we rephrase the identity obtained in terms of the $L$-function and the zeta functions of geodesic walks and galleries in

\begin{theorem}\label{buildingA2}
With $G=$PGL$_3$, the following identity for the finite quotient $\B_\Ga$ holds:
\begin{eqnarray}\label{A2-building}
 (1-u^3)^{\chi(\B_\Ga)}L(\Ga \bs G(F), \pi_1, qu) = \frac{Z(\B_\Ga, \pi_1, u) Z(\B_\Ga, \pi_2, u^2)}{Z_2(\B_\Ga, \pi_1, -u)},
\end{eqnarray}
where $\chi(\B_\Ga)$ denotes the Euler characteristic of $\B_\Ga$.
\end{theorem}

Let $N_0, N_1$, and $N_2$ denote the number of vertices, edges, and chambers in $\B_\Ga$, respectively. Then $N_1=(q^2+q+1)N_0$ and $N_2=
\frac{1}{3}(q+1)(q^2+q+1)N_0$ so that the Euler characteristic
$$\chi(\B_{\Ga})=N_0-N_1+N_2=\frac{1}{3}(q-1)^2(q+1)N_0.$$
Notice that when $q$ is set to $1$, $\chi(\B_{\Ga})$ vanishes and (\ref{A2-building})  reduces to (5.2) in Theorem \ref{Aidentity}.

\subsection{Type $\tilde C_2$}
Suppose now that $G=\PGSP_4$, then $\hat G(\mathbb C)=\Spin_5(\mathbb C)\simeq \Sp_4(\mathbb C)$. As in \S 5.2, denote by $\pi_\spin$ the degree $4$ spin representation of $\hat G(\mathbb C)$, which is minuscule, and by $\pi_\st$ the degree $5$ standard representation of $\hat G(\mathbb C)$, which is fundamental quasi-minuscule.
The weights in $\wt(\pi_\spin)$ can be identified with $\diag(1, 1, \varpi, \varpi)$, $\diag(\varpi, \varpi, 1, 1)$, $\diag(1, \varpi, 1, \varpi)$, and $\diag(\varpi, 1, \varpi, 1)$ in the standard apartment of $\B$,
while the non-trivial weights in $\wt'(\pi_\st)$ are identified with the matrices $\diag(\varpi^{-1}, 1, 1, \varpi)$, $\diag(\varpi, 1, 1, \varpi^{-1})$,
$\diag(1, \varpi^{-1}, \varpi, 1)$, and
$\diag(1, \varpi, \varpi^{-1}, 1)$.

Let $N_0, N_1$, and $N_2$ denote the number of vertices, edges, and chambers in $\B_\Ga$, respectively. It follows from Proposition 2.2 in
\cite{FLW} and the proof therein that $N_0=(q^2+3)N_p$, where $N_p$ counts the number of primitive special vertices in $\B_\Ga$. Then $N_1=3(q^3+q^2+q+1)N_p$ and as $\B_\Ga$ is a $(q+1)$-regular complex,
each edge is contained in $q+1$ chambers, so $N_2=(q^3+q^2+q+1)(q+1)N_p$.
The Euler characteristic of $\B_\Ga$ is
$$\chi(\B_{\Ga})=N_0-N_1+N_2=(q-1)^2(q^2+q+1)N_p.$$
In particular, it vanishes when $q$ is set to $1$.

The extended affine Weyl group $W_{ext}$ of $G$ is generated by 
$$s_0=\left(\begin{smallmatrix} &&& -\varpi^{-1}\\&1&&\\&&1&\\ \varpi&&& \end{smallmatrix}\right), s_1=\left(\begin{smallmatrix} &1&&\\1&&&\\&&&1\\&&1&\end{smallmatrix}\right), s_2=\left(\begin{smallmatrix} 1&&&\\&&1&\\&-1&&\\&&&1\end{smallmatrix}\right), \mbox{and }
\tau=\left(\begin{smallmatrix} &&1&\\&&&1\\ \varpi &&&\\&\varpi&&\end{smallmatrix}\right).$$

\subsubsection{Spin representation}
This case has been studied in detail in \cite{FLW}. The zeta identities are analogous to the PGL$_3$ case, but more complicated. We proceed to draw connections between zeta functions defined in \cite{FLW} and those defined in \S6.1.

A representation $\rho \in \Irr^\unr(\Ga \backslash G(F))$ is isomorphic to a subrepresentation of  a principal series $\chi_1 \times \chi_2 \rtimes  \sigma = \Ind \chi_\rho$ for some unramified characters $\chi_1, \chi_2, \sigma$  of $F^\times$ satisfying $\chi_1 \chi_2 \sigma^2 = 1$. The associated character $\chi_\rho$ sends an element $\diag(a, b, cb^{-1}, ca^{-1})$ in the maximal split torus of $G(F)$ to $\chi_1(a)\chi_2(b)\sigma(c)$. Therefore
$$\prod_{\la \in \wt(\pi_\spin)} (1 - \chi_\rho(\lambda)u) =  (1-\chi_1\sigma(\varpi)u)(1-\chi_2\sigma(\varpi)u)(1-\chi_1\chi_2\sigma(\varpi)u)(1-\sigma(\varpi)u).$$
It was shown in \cite{FLW} that the $L$-function of type $\pi_\spin$ can be expressed in terms of the Hecke operators $A_1$ and $A_2$ acting on the space $L^2(\Ga \bs G(F)/K)$ represented by the double cosets $K\diag(1, 1, \varpi, \varpi)K$ and $K\diag(\varpi^{-1}, 1, 1, \varpi)K$ plus $(q^2+1)$ times the identity operator, respectively, in the following way:
$$ L(\Ga \bs G(F), \pi_\spin, q^{3/2}u) = \frac{1}{\det(I-A_1 u+qA_2 u^2-q^3 A_1 u^3+q^6I u^4)}.$$
Moreover, the zeta function of $\pi_\spin$-geodesic walks agrees with the zeta function using only type $1$ edges in \cite{FLW} and hence can be expressed by the operator $L_{P_1}$ on $L^2(\Ga \bs G(F)/P_1)$ associated to a double coset of the Siegel subgroup $P_1$ of $K$:
$$ Z(\B_\Ga, \pi_\spin, u) = \frac{1}{\det(I-L_{P_1}u)}.$$
Similarly the zeta function of $\pi_\st$-geodesic walks agrees with the zeta function using only type $2$ edges in \cite{FLW} and
$$Z(\B_\Ga, \pi_\st, u) = \frac{1}{\det(I-L_{P_2}u)},$$
where $L_{P_2}$ is an operator on $L^2(\Ga \bs G(F)/P_2)$ associated to a double coset of the Klingen subgroup $P_2$ of $K$. Each $\pi_\spin$-chamber consists of two oriented chambers, with the orientation being the central directed edge of the $\pi_\spin$-chamber. The oriented chambers of $\B$ are parametrized by the cosets of the Iwahori subgroup $I$ in $G(F)$.
The Iwahori-Hecke operator $L_I$ on $L^2(\Ga \bs G(F)/I)$ associated to the double coset $ItI$ defined in \S3.3 of \cite{FLW} describes the adjacency of the $\pi_\spin$-chambers in $\B_\Ga$, where $t=s_0s_1\tau$.  As  observed in Hashimoto \cite{Ha1}, the trace of $L_I^n$ counts the number of closed $\pi_\spin$-geodesic galleries in $\B_\Ga$ of length $n$. Hence the zeta function of $\pi_\spin$-geodesic galleries of $\B_\Ga$ is given by
$$ Z_2(\B_\Ga, \pi_\spin, u) = \frac{1}{\det(I - L_Iu)}.$$

Analogous to the PGL$_3$ case, the zeta function of geodesic walks of $\B_\Ga$  counts the number of rational closed geodesic walks in $\B_\Ga$ in algebraic length; it is equal to
$$Z(\B_\Ga, u) = Z(\B_\Ga, \pi_\spin, u)Z(\B_\Ga, \pi_\st, u^2) = \frac{1}{\det(I-L_{P_1}u)\det(I-L_{P_2}u^2)}.$$
Theorem 1.1 of \cite{FLW} gives another expression of $Z(\B_\Ga, u)$ in terms of the Hecke operators $A_1$ and $A_2$ as well as the Iwahori-Hecke operator $L_I$. Combined, we obtain the zeta identity
$$\frac{(1-u^2)^{\chi(\B_{\Gamma})}(1-q^2 u^2)^{-(q^2-1)N_p}}{\det(I-A_1 u+qA_2 u^2-q^3 A_1 u^3+q^6I u^4)}
= \frac{\det(I-L_I u)}{\det(I-L_{P_1} u)\det(I-L_{P_2}u^2)}.$$
This is reinterpreted in terms of the Langlands $L$-function and the zeta functions of geodesic walks and galleries as

\begin{theorem} For $G =\PGSP_4$, there holds the identity for $\B_\Ga$:
\begin{eqnarray*}
(1-u^2)^{\chi(\B_{\Gamma})}(1-q^2 u^2)^{-(q^2-1)N_p}L(\Ga \bs G(F), \pi_\spin, q^{3/2}u) =
\frac{Z(\B_\Ga, \pi_\spin, u)Z(\B_\Ga, \pi_\st, u^2)}{Z_2(\B_\Ga, \pi_\spin, u)}.
\end{eqnarray*}
Here $N_p$ denotes the number of primitive special vertices in $\B_\Ga$, and $\chi(\B_{\Gamma})$ is the Euler characteristic of $\B_\Ga$.
\end{theorem}

Because the weights of $\pi_\spin$ connect vertices of different types, each closed $\pi_\spin$-geodesic gallery has even length,
so $Z_2(\B_\Ga, \pi_\spin, -u)=Z_2(\B_\Ga, \pi_\spin, u)$. When $q$ is set to $1$, the identity in Theorem 6.2 agrees with (\ref{spin}) in Theorem \ref{Aidentity}.

\subsubsection{Standard representation}
Identity (\ref{st}) in Theorem \ref{Aidentity} suggests the existence of another zeta identity that involves the degree five standard $L$-functions attached to unramified representations of $G(F)$. In this subsection we establish such an identity and compare it with (\ref{st}).

For a representation $\rho \in \Irr^\unr(\Ga \backslash G(F))$ isomorphic to a subrepresentation of a principal series $\chi_1 \times \chi_2 \rtimes  \sigma = \Ind \chi_{\rho}$ as before, we have
$$\prod_{\la \in \wt'(\pi_\st)} (1 - \chi_\rho(\lambda)u) =(1-\chi_1(\varpi) u)(1-\chi_2(\varpi)u)(1-\chi_1^{-1}(\varpi) u)(1-\chi_2^{-1}(\varpi)u).$$
Let $N$ be the cardinality of the set $\Irr^\unr(\Ga \backslash G(F))$, {then $N=\dim L^2(\Ga \bs G)^K=2N_p$}. Recall that $\wt(\pi_\st)$ is the union of $\wt'(\pi_\st)$ with the trivial weight and $\chi_\rho$ evaluated at the trivial weight is $1$, hence
$$ (1 - u)^N L(\Ga \bs G(F), \pi_\st, u) = \prod_{\rho \in \Irr^\unr(\Ga \backslash G(F))} \prod_{\la \in \wt'(\pi)} (1- \chi_\rho( \la) u )^{-1}.$$
To connect the $L$-function of type $\pi_\st$ with Hecke operators, we keep {$A_1, A_2$ the same as in the previous subsection, and introduce two variations of $A_2$, namely the operators $A_2'=A_2-
2q^2 I$ and $A_2'' = A_2 - q^2 I$.}  By the same argument as in \cite{FLW} one obtains
$$(1 - q^2u)^N L(\Ga \bs G(F), \pi_\st, q^2u) =\frac{1}{\det(I-A_2' u+(qA_1^2-2q^2A_2'')u^2-q^4 A_2' u^3+q^8I u^4)}.$$

Let $L_I'$ be the operator on $L^2(\Ga \bs \B/I)$ represented by the double coset $It'I$, where
$t'=\diag(1, \varpi, \varpi, {\varpi}^2) s_1 =s_0 s_1 s_2$.
The operator $L_I'$ defines adjacency of the $\pi_\st$-chambers in $\B_\Ga$. Then the trace of $(L_I')^n$ counts the number of closed $\pi_\st$-geodesic galleries in $\B_\Ga$ of length $n$.  We obtain
$$ Z_2(\B_\Ga, \pi_\st, u)=  \frac{1}{\det(I - L_I'u)}.$$

Consider a new zeta function on $\B_\Ga$ given by
$$Z'(\B_\Ga, u) = \frac{1}{\det(I-L_{P_1}u)^2 \det(I-L_{P_2}u)},$$
which counts closed $\pi_\spin$-geodesic walks in $\B_\Ga$ twice and closed $\pi_\st$-geodesic walks once. Since
each  closed $\pi_\spin$-geodesic walk has even length as the weights of $\pi_\spin$ connect vertices of different types, $\det(I-L_{P_1}u) = \det(I + L_{P_1}u)$ so that
$$Z'(\B_\Ga, u) = \frac{1}{\det(I-L_{P_1}^2u^2) \det(I-L_{P_2}u)} = Z(\B_\Ga, \pi_\st, u)Z(\B_\Ga, \pi_\spin, u)^2.$$

We proceed to derive a second description of $Z'(\B_\Ga, u)$ in terms of the operators $A_1, A_2', A_2'',$ and $L_I'$.

Since $\Gamma$ is a discrete cocompact subgroup of $G(F)$, $L^2(\Gamma \bs G(F))$ decomposes into a direct sum of irreducible unitary representations. To study actions of the operators introduced above, we are only concerned with Iwahori-spherical representations. These representations are classified into types and described in Table 1 below, which is adapted from \cite{Sc}. Table 1 also records the dimensions of certain invariant subspaces. The column C counts the number
$$4\dim V^K +\dim V^{I}-(2 \dim V^{P_1}+\dim V^{P_2})$$ for each type of representation $(\rho, V)$.

\begin{center}
\begin{tabular}{|c|c|c|c|c|c|c|c|} \hline
\text{type} &\text{representation } $(\rho, V)$ &$V^K$  &$V^{P_2}$ &$V^{P_1}$ &$V^I$&&$C$ \\
\hline
% I
   {\rm I}&$\chi_1\times\chi_2\rtimes\sigma$
   &1&4&4&8  &&0   \\
   \hline
%II
   {\rm IIa}&$\chi\St_{\GL_2}\rtimes\sigma$&0&2&1&4   &&0    \\
   {\rm IIb}&$\chi\triv_{\GL_2}\rtimes\sigma$&1&2&3&4   &&0   \\
   \hline
%III
   {\rm IIIa}&$\chi\rtimes\sigma\St_{\GSp_2}$&0&1&2&4   &&-1   \\
   {\rm IIIb}&$\chi\rtimes\sigma\triv_{\GSp_2}$&1&3&2&4   &&1  \\
   \hline
%IV
   {\rm IVa}&$\sigma\St_{\GSp_4}$&0&0&0&1     &&1    \\
   {\rm IVd}&$\sigma\triv_{\GSp_4}$&1&1&1&1     &&2   \\
   \hline
%V
   {\rm Va}&$\delta([\xi,\nu\xi],\nu^{-\frac{1}{2}}\sigma)$&0&1&0&2     &&1    \\
   {\rm Vb}&$L(\nu^{\frac{1}{2}}\xi\St_{\GL_2},\nu^{-\frac{1}{2}}\sigma)$&0&1&1&2   &&-1  \\
   {\rm Vc}&$L(\nu^{\frac{1}{2}}\xi\St_{\GL_2},\xi\nu^{-\frac{1}{2}}\sigma)$&0&1&1&2    &&-1   \\
   {\rm Vd}&$L(\nu\xi,\xi\rtimes\nu^{-\frac{1}{2}}\sigma)$&1&1&2&2    &&1   \\
   \hline
%VI
   {\rm VIa}&$\tau(S,\nu^{-\frac{1}{2}}\sigma)$&0&1&1&3    &&0   \\
   {\rm VIb}&$\tau(T,\nu^{-\frac{1}{2}}\sigma)$&0&0&1&1    &&-1  \\
   {\rm VIc}&$L(\nu^{\frac{1}{2}}\St_{\GL_2},\nu^{-\frac{1}{2}}\sigma)$&0&1&0&1    &&0   \\
   {\rm VId}&$L(\nu,\triv_{F^\times}\rtimes\nu^{-\frac{1}{2}}\sigma)$&1&2&2&3   &&1   \\
   \hline

\end{tabular}
\end{center}
\begin{center}
Table 1
\end{center}

Let $m_i$ denote the number of type $i$ representations in $L^2(\Ga \bs G)$, counting multiplicity.
We are now ready to state the main result of this section.

\begin{theorem} \label{deg5}
The following two expressions for $Z'(\B_\Ga, u)$ hold.
\begin{align*}
Z'(\B_{\Gamma} ,u)&=\frac{1}{{\det (I-L_{P_1}^2 u^2)}{\det (I-L_{P_2}u)}}\\
&=\frac{(1-u)^{2\chi(\B_\Ga)}(1-qu)^{t_1}(1+qu)^{t_2}}{\det(I-A_2' u+(qA_1^2-2q^2A_2'')u^2-q^4 A_2' u^3+q^8I u^4)\det(I+L_I' u)},
\end{align*}
where $t_1=-m_{IIIa}+m_{IIIb}+m_{IVd}-m_{VIb}+m_{VId}$, and $t_2=m_{Va}-m_{Vb}-m_{Vc}+m_{Vd}$ with $t_1+t_2=-2(q^2-1)N_p$.
\end{theorem}

We follow the same approach as in \cite{FLW} computing eigenvalues of $-L_I'$, $\pm L_{P_1}$,
$L_{P_2}$ on each type of representation $(\rho, V)$, and record the data in Table 2.
The column $A_1, A_2', A_2''$ records the reciprocals of the zeros of $$\det(I-A_2' u+(qA_1^2-2q^2A_2'')u^2-q^4 A_2' u^3+q^8I u^4)$$ arising from $\rho$. The reader is referred to the of allcillary in \cite{KLW2} for details.

\begin{center}
\resizebox{13cm}{!}{
\begin{tabular}{|c|c|c|c|c|c|}
\hline
type& $-L_I'$& $\pm L_{P_1}$&  $L_{P_2}$ & $A_1$, $A_2', A_2''$ & condition\\
\hline
I & $\pm q^{\frac{3}{2}} \chi_1 \sigma(\pi)$  & $\pm q^{\frac{3}{2}}\chi_1\sigma(\pi)$ & $q^{2}\chi_1 \sigma^2 (\pi)$ & $q^{2}\chi_1\sigma^2(\pi)$ & $\chi_1\chi_2\sigma^2 = \bf 1$\\
  \  & $\pm q^{\frac{3}{2}}\chi_{2}\sigma(\pi)$ & $\pm q^{\frac{3}{2}}\chi_{2}\sigma(\pi)$ & $q^{2}\chi_{2}\sigma^2(\pi)$ &  $q^2\chi_{2}\sigma^2(\pi)$& \\
  \  & $\pm q^{\frac{3}{2}}\chi_{1}\chi_2\sigma(\pi)$ & $\pm q^{\frac{3}{2}}\chi_{1}\chi_{2}\sigma(\pi)$ & $q^{2}\chi_{1}\chi_2^2 \sigma^2(\pi)$ & $q^2\chi_{1}\chi_{2}^2\sigma^2(\pi)$ &\\
  \  & $\pm q^{\frac{3}{2}}\sigma(\pi)$ & $ \pm q^{\frac{3}{2}}\sigma(\pi)$ & $q^{2}\chi_1^2\chi_{2}\sigma^2(\pi)$ & $q^2\chi_1^2\chi_2\sigma^2(\pi)$& \\
 \hline

IIa & $\pm q $ & $ \pm q$ & $q^{\frac{3}{2}}\chi^{-1}(\pi)$ & none & $\chi^2\sigma^2 = \bf 1$\\
 \  & $q^{\frac{3}{2}}\chi(\pi), q^{\frac{3}{2}}\chi^{-1}(\pi)$ & \ & $ q^{\frac{3}{2}}\chi(\pi)$ & \ & \\
 \hline
IIb & $\pm q^2$ & $\pm q^{\frac{3}{2}}\chi^2\sigma(\pi)$ & $q^{\frac{5}{2}}\chi(\pi)$ &  $q^{\frac{5}{2}}\chi(\pi)$ & $\chi^2\sigma^2 = \bf 1$ \\
 \ & $-q^{\frac{3}{2}}\chi(\pi), -q^{\frac{3}{2}}\chi^{-1}(\pi)$ & $\pm q^{\frac{3}{2}}\sigma(\pi)$ & $q^{\frac{5}{2}}\chi^{-1}(\pi)$ &  $q^{\frac{5}{2}}\chi^{-1}(\pi)$  &\\
 \ & \  & $\pm q^{2}$ & \  & $q^{\frac{3}{2}}\chi(\pi), q^{\frac{3}{2}}\chi^{-1}(\pi)$ &\\
 \hline

IIIa & $\pm q\chi\sigma(\pi)$ & $\pm q\chi \sigma(\pi)$ & $q$ & none & $\chi \sigma^2 = \bf 1$\\
 \  & $\pm q\sigma(\pi)$ & $\pm q\sigma(\pi)$ & \  & \ & \\
 \hline
IIIb & $\pm q^{2}\chi\sigma(\pi)$ & $\pm q^{2}\chi\sigma (\pi)$ & $q^{3}$ & $q, q^3 $ & $\chi \sigma^2 = \bf 1$ \\
 \  & $\pm q^2\sigma(\pi)$ & $\pm q^{2}\sigma(\pi)$ & $q^{2}\chi(\pi), q^{2}\chi^{-1}(\pi)$  & $q^{2}\chi(\pi), q^2\chi^{-1}(\pi)$ & \\
 \hline

IVa & $1$  & none & none & none & $\sigma^2 = \bf 1$\\
\hline
IVd &  $-q^{3}$ & $\pm q^{3}$  & $q^{4}$ & $q^{4}, q^{3}, q, 1$& $\sigma^2 = \bf 1$\\
\hline

Va & $-q, -q$ & none & $-q$ & none & $\sigma^2 = \bf 1$\\
\hline
Vb & $-q^2, q$ & $\pm q$ & $-q^{2}$ & none & $\sigma^2 = \bf 1$\\
\hline
Vc & $-q^{2}, q$ & $\pm q$ & $-q^{2}$ & none & $\sigma^2 = \bf 1$\\
\hline
Vd & $q^2, q^2$ & $\pm q^{2}, \pm q^2$ & $-q^{3}$ & $-q^{3}, -q^2, -q^2, -q$ & $\sigma^2 = \bf 1$\\
\hline

VIa & $\pm q, q$ & $\pm q$ & $q$ & none & $\sigma^2 = \bf 1$\\
\hline
VIb & $-q$ & $\pm q$ & none & none & $\sigma^2 = \bf 1$\\
\hline
VIc & $q^2$ & none & $q^{2}$ & none & $\sigma^2 = \bf 1$\\
\hline
VId & $q^2, -q^2, -q^2$ & $\pm q^{2}, \pm q^{2}$& $q^{3}, q^{2}$ & $q^{3}, q^{2}, q^2, q$ & $\sigma^2 = \bf 1$\\
\hline

\end{tabular}
}
\end{center}

\begin{center}
Table 2
\end{center}

\begin{proof}

The zeta identity follows from Table 2. To show $t_1+t_2=-2(q^2-1)N_p$, observe from Table 1 that
\begin{multline*}
4\dim L^2(\Ga \bs G)^K+\dim L^2(\Ga \bs G)^I-(2\dim L^2(\Ga \bs G)^{P_1}+\dim L^2(\Ga \bs G)^{P_2})\\
=-m_{IIIa}+m_{IIIb}+m_{IVa}+2m_{IVd}+m_{Va}-m_{Vb}-m_{Vc}+m_{Vd}-m_{VIb}+m_{VId}.
\end{multline*}

Note that $\dim L^2(\Ga \bs G)^K=2N_p=\frac{2}{q^2+3}N_0$ and $\dim L^2(\Ga \bs G)^I=2N_2$ as the
former counts the number of special vertices and the latter counts the number of directed chambers in $B_{\Ga}$. As $\dim L^2(\Ga \bs G)^{P_1}$ and $\dim L^2(\Ga \bs G)^{P_2}$ count the number of directed type 1 edges and type 2 edges, respectively, they are both equal to $2(q^3+q^2+q+1)N_p$, and we have $$2\dim L^2(\Ga \bs G)^{P_1}+\dim L^2(\Ga \bs G)^{P_2}=2N_1.$$ So the left hand side of the identity equals $$\frac{8}{q^2+3}N_0-2N_1+2N_2.$$

Since $m_{IVa}=2(\chi(B_\Ga)-1)$ and $m_{IVd}=2$, the right hand side of the identity equals $$t_1+t_2-2+2(\chi(B_{\Ga})-1)+4=t_1+t_2+2\chi(B_{\Ga}).$$

\noindent Comparing the two sides gives $t_1+t_2=-2(q^2-1)N_p.$

\end{proof}

We rephrase Theorem \ref{deg5} in terms of the Langlands $L$ and zeta functions of geodesic walks and galleries. Put
$E(u, q)=(1-u)^{2\chi(\B_{\Gamma})}(1-q u)^{t_1} (1+q u)^{t_2}$, where
$t_1=-m_{IIIa}+m_{IIIb}+m_{IVd}-m_{VIb}+m_{VId}$, $t_2=m_{Va}-m_{Vb}-m_{Vc}+m_{Vd}$ with $t_1+t_2 = -2(q^2-1)N_p$ are as in Theorem \ref{deg5}.

\begin{theorem}\label{st-building}
For $G=\PGSP_4$, the following identity holds:

\begin{eqnarray*}
E(u, q)(1-q^2 u)^N L(\Ga \bs G(F), \pi_\st, q^2 u)
=\frac{Z(\B_\Ga, \pi_\st, u)Z(\B_\Ga, \pi_\spin, u)^2}{Z_2(\B_\Ga, \pi_\st, -u)},
\end{eqnarray*}
where $E(u, q)$ is as defined above, and $N$ is the number of special vertices  in $\B_\Ga$.
\end{theorem}

Leaving aside the factor $E(u, q)$, we note that, upon letting $q=1$, the remaining part reduces to the identity (\ref{st}).
\bigskip

\noindent {\bf Acknowledgment.} The authors are grateful to the anonymous referee for the valuable comments and corrections which clarified and improved the exposition of this paper.

\end{document}